\documentclass[11pt]{amsart}
\usepackage{microtype,fullpage,amsmath,amsthm}
\usepackage{graphicx,setspace,amscd,float,hyperref,enumitem, bbold}

\newtheorem{thm}{Theorem}[section]
\newtheorem{lem}[thm]{Lemma}
\newtheorem{qst}[thm]{Question}
\newtheorem{prop}[thm]{Proposition}

\theoremstyle{definition}
\newtheorem{df}[thm]{Definition}
\newtheorem{rk}[thm]{Remark}
\newtheorem{ex}[thm]{Example}
\newtheorem{nt}[thm]{Notation}
\newtheorem{con}[thm]{Convention}
\newtheorem*{mainthm}{Theorem}

\newcommand{\RR}{\mathbb R}
\newcommand{\NN}{\mathbb N}
\newcommand{\Gam}{\Gamma}

\newcommand{\ol}{\overline}
\newcommand{\mG}{\mathcal{G}}

\newcommand{\mC}{\mathcal{C}}
\newcommand{\mE}{\mathcal{E}}

\newcommand{\mD}{\mathcal{D}}
\newcommand{\mT}{\mathcal{T}}
\newcommand{\mP}{\mathcal{P}}

\newcommand{\mW}{\mathcal{W}}
\newcommand{\EG}{\mathcal{E}(\Gamma)}
\newcommand{\iwp}{\mathcal{IW}(\phi)}
\newcommand{\swg}{\mathcal{SW}(g)}
\newcommand{\lwg}{\mathcal{LW}(g)}
\newcommand{\mcg}{\mathcal{C}(G)}

\newcommand{\cwg}{\mathcal{CW}(g)}
\newcommand{\id}{\mathcal{ID}}
\newcommand{\iw}{\mathcal{IW}}
\newcommand{\idg}{\mathcal{ID}(\mathcal{G})}
\newcommand{\ar}{\mathcal{A}_r}
\newcommand{\phar}{\phi \in \mathcal{A}_r}

\newcommand{\outr}{Out(F_r)}
\newcommand{\autr}{Aut(F_r)}

\begin{document}

\title{Ideal Whitehead Graphs in $\outr$ III:\\ Achieved Graphs in Rank 3}
\author{Catherine Pfaff}
\date{}
\maketitle

\begin{abstract}
By proving precisely which singularity index lists arise from the pair of invariant foliations for a pseudo-Anosov surface homeomorphism, Masur and Smillie \cite{ms93} determined a Teichm\"{u}ller flow invariant stratification of the space of quadratic differentials. In this final paper of a three-paper series, we give a first step to an $Out(F_r)$ analog of the Masur-Smillie theorem. Since the ideal Whitehead graphs defined by Handel and Mosher \cite{hm11} give a strictly finer invariant in the analogous $Out(F_r)$ setting, we determine which of the twenty-one connected, simplicial, five-vertex graphs are ideal Whitehead graphs of fully irreducible outer automorphisms in $Out(F_3)$. 
\end{abstract}

\section{Introduction}

Let \emph{$F_r$} denote the free group of rank $r$ and $Out(F_r)$ its outer automorphism group.  In this paper we prove realization results for an invariant dependent only on the conjugacy class (within $Out(F_r)$) of the outer automorphism, namely the ``ideal Whitehead graph.'' 

\subsection{Main result}{\label{ss:mr}}

A ``fully irreducible'' (iwip) outer automorphism is the most commonly used analogue to a pseudo-Anosov mapping class and is generic. An element $\phi \in Out(F_r)$ is \emph{fully irreducible} if no positive power $\phi^k$ fixes the conjugacy class of a proper free factor of $F_r$. 

We give in Subsection \ref{ss:iwg} the exact $Out(F_r)$ definition of an ideal Whitehead graph. For now, to give context, we remark that, for a pseudo-Anosov surface homeomorphism, the component of an ideal Whitehead graph coming from a foliation singularity is a polygon with edges corresponding to the lamination leaf lifts bounding a principal region in the universal cover \cite{n86}.

Handel and Mosher define in \cite{hm11} a notion of an ideal Whitehead graph for a fully irreducible outer automorphism, a finite graph whose isomorphism type is an invariant of the conjugacy class of the outer automorphism. In this paper we investigate the extent to which the $Out(F_r)$ situation is more complicated by giving a partial answer to a question posed by Handel and Mosher in \cite{hm11}:

\begin{qst}{\label{Q:Q1}} For each $r \geq 2$, which isomorphism types of graphs occur as $\mathcal{IW}(\phi)$ for a fully irreducible $\phi \in Out(F_r)$? \end{qst}

\begin{mainthm} {\bf{A.}}
Exactly eighteen of the twenty-one connected, simplicial five-vertex graphs are the ideal Whitehead graph $\mathcal{IW}(\phi)$ for a fully irreducible outer automorphism $\phi \in Out(F_3)$.
\end{mainthm}

\noindent The twenty-one connected, simplicial five-vertex graphs (\cite{cp84}) are:
~\\
\vspace{-6.5mm}
\begin{figure}[H]
\centering
\includegraphics[width=3.7in]{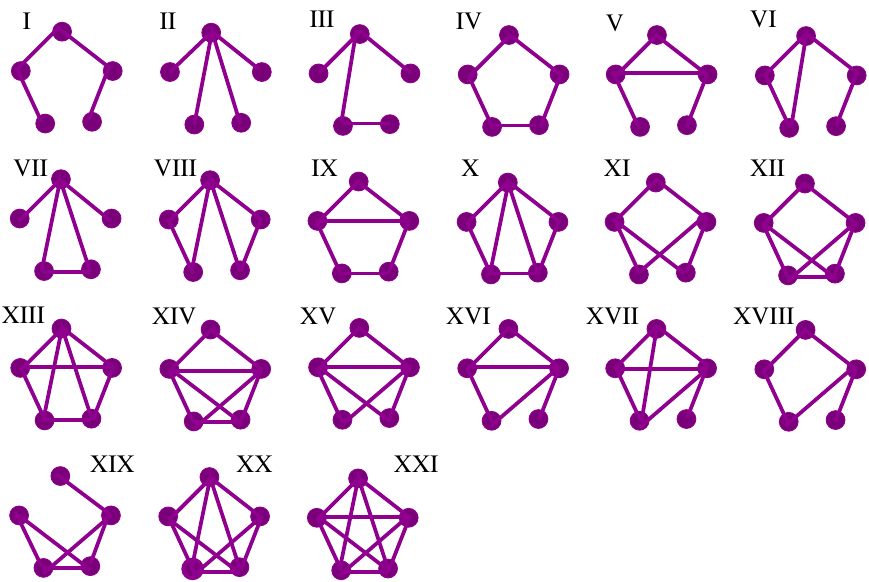}
\label{fig:21Graphs}
\end{figure}

\noindent Those that are not the ideal Whitehead graph for any fully irreducible $\phi\in Out(F_3)$ are:
\begin{figure}[H]
\centering
\includegraphics[width=2.3in]{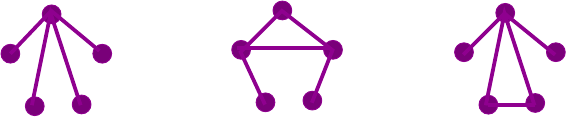}
\label{fig:UnachievableGraphs} 
\end{figure}

We focus on the situation of graphs with $2r-1$ edges because this restricts our attention to fully irreducibles as close to those coming from surface homeomorphisms (\emph{geometrics}) as possible without actually coming from surface homeomorphisms (or being geometric-like \emph{parageometrics}).

\subsection{An ideal Whitehead graph definition}{\label{ss:iwg}}

Fully irreducible outer automorphisms have ``train track representatives'' \cite{bh92}: Let $R_r$ denote the $r$-petaled rose (graph with one vertex and $r$ edges) together with an identification $\pi(R_r) \cong F_r$. A finite connected $1$-dimensional CW-complex $\Gamma$ such that each vertex has valence greater than two, together with a homotopy equivalence (\emph{marking}) $R_r \to \Gamma$, is called a \emph{marked graph}. A \emph{train track (tt) representative} of an outer automorphism $\phi \in Out(F_r)$ is a homotopy equivalence $g \colon \Gamma \to \Gamma$ of a marked graph $\Gamma$, where 

\begin{itemize}
\item $\phi=g_*$,
\item $g$ sends vertices to vertices, and
\item $g^k$ is locally injective on edge interiors for each $k>0$.
\end{itemize}

The ideal Whitehead graph was defined by Handel and Mosher in \cite{hm11} using the expanding lamination for a fully irreducible outer automorphism. We give an equivalent (see \cite{pt}) combinatorial description here, via the local Whitehead graphs and local stable Whitehead graphs of \cite{hm11} (interesting on their own right, as described in \ref{ss:context}):

Given a marked graph $\Gamma$ and point $x \in \Gamma$, a \emph{direction} at $x$ is a germ of edge-segments emanating from $x$. To avoid over-burdening the reader with notation, for an edge $e$, we additionally let $e$ denote the germ of initial segments of $e$ and let $\bar{e}$ denote the germ of terminal segments of $e$. (For edges, $\bar{e}$ denotes the edge $e$ with the reverse orientation.) A map $g \colon \Gamma \to \Gamma$ induces a map of directions, which we denote by $Dg$. A direction $d$ is called \emph{periodic} for $g$ if $Dg^k(d)=d$ for some $k>0$.

Let $\Gamma$ be a marked graph, $v \in \Gamma$ a vertex with at least three periodic directions, and $g\colon \Gamma \to \Gamma$ a train track representative of $\phi \in Out(F_r)$. The \emph{local Whitehead graph} $\mathcal{LW}(g; v)$ for $g$ at $v$ has:

(1) a vertex for each direction $d$ at $v$ and

(2) an edge connecting the vertex for $e_i$ with the vertex for $e_j$ when there exists an edge $E$ of $\Gamma$ and $k > 0$ such that $g^k(E)$ contains either $\bar{e_i}{e_j}$ or $\bar{e_j}{e_i}$ ($g^k$ \emph{takes the turn} $\{e_i,e_j\}$).

\noindent The \emph{local stable Whitehead graph} $\mathcal{SW}(g; v)$ is the subgraph of $\mathcal{LW}(g; v)$ obtained by restricting precisely to vertices labeled by periodic directions and the edges connecting them. Given an appropriate choice of train track representative (see Section \ref{Ch:PrelimDfns} for a detailed explanation), the \emph{ideal Whitehead graph $\mathcal{IW}(\phi)$ of $\phi$} will precisely be a disjoint union of the local stable Whitehead graphs at the periodic vertices. (In a nonideal situation, some gluing of pairs of vertices in distinct local stable Whitehead graphs will need to occur to obtain the ideal Whitehead graph.)

\subsection{Context}{\label{ss:context}}

Free groups and their outer automorphisms are studied via a variety of Whitehead graphs, linking purely algebraic, dynamical, and topological properties.  

Whitehead first introduced Whitehead graphs \cite{w36} in the context of solving the automorphic equivalence problem for free groups, i.e the question as to whether two elements $u,v\in F_r$ satisfy that $\phi(u)=v$ for some $\phi\in Aut(F_r)$. This led to the Whitehead algorithm of \cite{w36}, used to determine whether or not a word in the free group, written in a given basis, represents a primitive element of the free group. There is prolific work surrounding the Whitehead algorithm. Whitehead graphs also play a key role in the proof by Culler and Vogtmann \cite{cv86} of the contractibility of Culler-Vogtmann outer space.

The study of local Whitehead graphs, local stable Whitehead graphs, and ideal Whitehead graphs is much newer and carries more topological and dynamical information about outer automorphisms.

The local Whitehead graph is used, for example, in the Full Irreducibility Criterion of \cite{pI}. And in fact, it has been known for many years that full irreducibility requires that no representative has a local Whitehead graph with separating vertices.

Ideal Whitehead graphs and stable Whitehead graphs also tell us about fold lines in Culler-Vogtmann outer space and train track maps. Fold lines are of importance in understanding the geometry of outer space, as (apart from geodesics formed by shrinking the edge of a graph) geodesics of outer space are fold lines. (An explanation of how fold lines are geodesics can be found in \cite{fm11} or \cite{b12}.) Handel and Mosher proved in \cite{hm11} that, for a fully irreducible $\phi$, the ability to fold from a point in outer space (marked graph with lengths on edges) carrying a train track representative of a $\phi^k$ to a distinct such point, depends on a decomposition of the ideal Whitehead graph for $\phi$ into local stable Whitehead graphs. The decomposition of an ideal Whitehead graph into local stable Whitehead graphs further gives important information about the nature of the train track representatives of a fully irreducible. The possible decompositions of the ideal Whitehead graph are exactly determined by the separating vertices in the ideal Whitehead graph.

It is notable that, while local Whitehead graphs and stable Whitehead graphs record finer information about train track representatives, the ideal Whitehead graph is the only one of these representative invariants to in fact only rely on the expanding lamination of a fully irreducible (and hence to actually be an outer automorphism invariant).

\subsubsection{Ideal Whitehead graphs and laminations}{\label{sss:laminations}}

In \cite{bfh97}, Bestvina, Feighn, and Handel define the expanding lamination for a fully irreducible outer automorphism. This lamination is an analog of the expanding lamination for a pseudo-Anosov. However, as explained in \cite{chl08I} and \cite{abhs06}, the lamination can also be realized as the fixed point of a substitution system in the context of substitution systems and symbolic dynamics. In all of these contexts, the lamination is not only a set of lines invariant under the action of a fully irreducible, but also satisfies that all other lines are attracted to it under the action. We do not give or use a formal description of the expanding lamination in this paper, but one way to quickly understand the intrinsic significance of the object is to view it as the set of infinite words (or pairs of distinct points in the boundary of the free group) obtained by applying the automorphism repeatedly to a free group generator and then taking the closure in the space of lines. One can find an explanation in \cite{hm11} or \cite{pt} of how the edges of the ideal Whitehead graph come from leaves of the expanding lamination. 

\subsubsection{Relationships to mapping class groups and index theory}{\label{sss:index}}

Index theory and the index list realization work of Masur and Smillie in \cite{ms93} provide one large source of motivation for exploration of ideal Whitehead graph realization in the $Out(F_r)$ context. Index theory exists both in the mapping class group (pseudo-Anosov) setting and in the $Out(F_r)$ setting. We start by explaining the mapping class group context for those more familiar with surface theory.

Recall that the \emph{mapping class group} of a compact surface $S$ is the group of isotopy classes of orientation-preserving self-homeomorphisms of $S$. Pseudo-Anosovs are the most common mapping class group elements (see for example \cite{m11}) and are characterized by having a representative leaving invariant a pair of transverse measured singular minimal foliations. The index list can identically be ascertained from the invariant foliation of the pseudo-Anosov or from its dual $\RR$-tree. In fact, the singularities of the invariant foliation, lifted to the universal cover, are in one-to-one correspondence with the branchpoints of the dual $\RR$-tree. In the respective settings, the index list has an entry of $1-\frac{k}{2}$ obtained by counting the number $k$ of prongs at the singularity or the valence of the branch point. For a pseudo-Anosov, the component of the ideal Whitehead graph coming from a foliation singularity is a polygon with edges corresponding to the lamination leaf lifts bounding a principal region in the universal cover \cite{n86}. In fact, since the number of vertices of each polygonal ideal Whitehead graph component is determined by the number of prongs of the singularity, the index list and the ideal Whitehead graph record the same data in this setting. (For each component of the ideal Whitehead graph with $k$ vertices, the index list has an entry $1-\frac{k}{2}$.) Note that, alternatively, one can ascertain the index list (thus ideal Whitehead graph) from singularities of the expanding invariant lamination (obtained as the limit of any simple closed curve under repeated application of the pseudo-Anosov) or from the invariant train track. 

A key point about index lists in the pseudo-Anosov setting is that they determine a stratification (invariant under the Teichm\"{u}ller flow) of the cotangent bundle of a Teichm\"{u}ller space. The flow is proved ergodic and mixing by Masur \cite{m82} and Veech \cite{v82}. The stratification itself has even been extensively studied, for example in the work of \cite{emr12}, \cite{kz03}, \cite{l04}, \cite{l05}, and \cite{z10}. Our work particularly relates to the work on determining connected components of a strata, as we show that a given index list naturally divides into many components. However, the theorem of Masur and Smillie \cite{ms93}, we directly strive to emulate in this paper, lists precisely which singularity index lists permitted by the Poincar\'e-Hopf index formula are realized by pseudo-Anosovs. The author hopes this work is a step toward having the machinery to emulate, in the $Out(F_r)$ setting, theorems about the dynamics of Teichm\"{u}ller space with the Teichm\"{u}ller metric.

As with a pseudo-Anosov acting on Teichm\"{u}ller space (or a hyperbolic isometry acting on hyperbolic space), a fully irreducible acts with north-south dynamics \cite{ll03} on the natural compactification of Culler-Vogtmann outer space \cite{cv86}. Both the attracting and repelling points for the action are $\RR$-trees, denoted respectively $T^+_{\phi}$ and $T^-_{\phi}$. The repelling tree is an extension, to fully irreducibles not induced by pseudo-Anosovs, of the dual tree to the invariant foliation for a pseudo-Anosov. As with a pseudo-Anosov, the index list for a fully irreducible, as defined in \cite{hm11}, has an entry of $1-\frac{k}{2}$ obtained by counting the valence $k$ of the branch point. The index list can again also be computed from the expanding lamination of \cite{bfh97}. The key observation here (explored in this paper) is the fact that the singular lamination leaves need not simply connect adjacent periodic points in the boundary of $T^+_{\phi}$, as would make the situation analogous to that in the mapping class group setting. In other words, the components of the ideal Whitehead graph need not just be polygons. It is also important to note that, while the indices for a pseudo-Anosov on a surface sum to the Euler characteristic of the surface, the index sum for a fully irreducible in $Out(F_r)$ is not a function of $r$, but is uniformly bounded by a function of $r$, as proved in \cite{gjll}.

\subsubsection{Applications of ideal Whitehead graphs}{\label{sss:apps}}

In \cite{hm11} Mosher and Handel defined the axis bundle for a fully irreducible and use ideal Whitehead graphs to prove that the axis bundle for each nongeometric fully irreducible outer automorphism is proper homotopy equivalent to a line. The axis bundle is an analog to the axis for a hyperbolic isometry of hyperbolic space or the Teichm\"{u}ller axis for a pseudo Anosov. It has further significance in that, if $\phi$ and $\psi$ are two fully irreducible outer automorphisms, then $\mathcal{A}_{\phi}$ and $\mathcal{A}_{\psi}$ differ by the action of an $Out(F_r)$ on its Culler-Vogtmann outer space if and only if there exist powers $k,l>0$ such that $\phi^k$ and $\psi^l$ are conjugate in $Out(F_r)$. 

More recently Mosher and Pfaff \cite{mp13} proved a necessary and sufficient ideal Whitehead graph condition for an axis bundle to be a unique, single axis. The essence of this fact is that the axis bundle for a fully irreducible $\phi$ is the closure of the set of points in outer space (marked graphs with lengths on edges) on which there exists an affine train track representative for some power $\phi^k$ of $\phi$. And, as mentioned above, the decomposition of an ideal Whitehead graph into local stable Whitehead graphs gives information about the nature of the train track representatives of a fully irreducible, as well as about when two points in outer space are connected by a fold line.

Finally, the index list for a nongeometric fully irreducible can be computed from the number of vertices in the components of the ideal Whitehead graph, exactly as in the pseudo-Anosov case.

\subsection{Related work}{\label{ss:rw}}

While exploration into understanding ideal Whitehead graphs is still a young endeavour, the index theory for free group outer automorphisms has in some directions already been extensively developed. In fact, there are three types of $Out(F_r)$ index invariants in the literature, those of \cite{gl95}, \cite{gjll}, and \cite{ch10}. The index of $\phi$, as defined and studied in \cite{gjll}, is equal to the geometric index of $T^+_{\phi}$, as established by Gaboriau-Levitt \cite{gl95} for more general $\RR$-trees. \cite{ch12} provides a relationship between the index of \cite{ch10} and the geometric index, as well as uses the index to relate different properties of the attracting and repelling tree for a fully irreducible. There are also even index realization results of several different natures. For example, \cite{jl09} gives examples of automorphisms with the maximal number of fixed points on $\partial F_r$, as dictated by an inequality in \cite{gjll}. And \cite{p13} gives a version of the Masur-Smillie theorem for fully irreducibles in $Out(F_3)$, focusing on a related inequality in \cite{gjll}. 

In \cite{pI} we proved that certain types of graphs are never ideal Whitehead graphs. We proved in \cite{pII} that in each rank $r$, the complete ($2r-1$)-vertex graph is the ideal Whitehead graph for some fully irreducible $\phi \in Out(F_r)$. 

\subsection*{Acknowledgements}

\indent The author expresses gratitude to Lee Mosher for incredible patience and generosity, invaluable discussions, and answers to her endless questions; Mark Feighn for numerous meetings and question answers, as well as recommendations for uses of her methods; Michael Handel for meeting with her and recommendations; Yael Algom-Kfir and Mladen Bestvina for all they patiently taught her; Arnaud Hilion and Ilya Kapovich for their assistance and advice; and Martin Lustig for his continued interest in her work. She also extends her gratitude to Bard College at Simon's Rock and the CRM for their hospitality.

\section{Preliminary definitions and notation}{\label{Ch:PrelimDfns}}

For each graph $\mathcal{G}$ of Theorem \ref{T:MainTheorem}, we construct a train track representative (in the sense of \cite{bh92}) of a fully irreducible $\phi \in Out(F_r)$ with ideal Whitehead graph $\mathcal{IW}(\phi) \cong \mathcal{G}$. Subsection \ref{ss:tt} is thus devoted to establishing the definitions and notation we use regarding train track maps. The existence of ``ideally decomposed'' representatives was established in \cite{pII}, see Subsection \ref{ss:id}, allowing us to restrict attention to maps on roses. In \cite{pII}, we introduced ltt structures (described in Subsection \ref{ss:ltt}), graphs extending a potential ideal Whitehead graph to include the edges of the rose, as well as an edge including the nonfixed direction of the train track representative. An ``admissible fold'' (in practice an inverse of a fold) yields a move between ltt structures, either an ``extension'' or ``switch,'' as described in Subsection \ref{ss:se}. For a potential ideal Whitehead graph $\mathcal{G}$, an ``$\mathcal{ID}$ diagram'' is defined in \cite{pI} whose vertices are ltt structures, whose directed edges are admissible moves, and who contain loops for train track representatives of fully irreducible $\phi \in Out(F_r)$ with $\mathcal{IW}(\phi) \cong \mathcal{G}$. We describe $\id$ diagrams in Subsection \ref{ss:idd}. In Section \ref{s:Procedure} we describe methods for composing folds (within an $\mathcal{ID}$ diagram) to obtain an irreducible train track map whose ideal Whitehead graph will be the desired graph $\mathcal{G}$.

Given a rank $r \geq 2$, we let $\mathcal{FI}_r$ denote the set of all fully irreducible $\phi \in Out(F_r)$.

\subsection{Graph maps and train track maps}{\label{ss:tt}}
We first recall from the introduction the definition of a ``train track representative'' for a $\phi \in Out(F_r)$ (as defined in \cite{bh92}). 

\begin{df}[Train track (tt) representatives] 
Let $R_r$ denote the $r$-petaled rose (graph with one vertex and $r$ edges) together with an identification $\pi(R_r) \cong F_r$. A connected $1$-dimensional CW-complex $\Gamma$ such that each vertex has valence greater than two, together with a homotopy equivalence (\emph{marking}) $R_r \to \Gamma$, is called a \emph{marked graph}. A \emph{train track (tt) representative} of $\phi$ is a homotopy equivalence $g \colon \Gamma \to \Gamma$ of a marked graph $\Gamma$, where $\phi=g_*$, where $g$ sends vertices to vertices, and where $g^k$ is locally injective on edge interiors for each $k>0$.
\end{df}

We need a more general notion of a ``graph map'' and ``train track map''. Every map we deal with in this paper is on the rose. Thus, we assume $\Gamma$ is a rose (wedge of circles), with vertex $v$, and give all definitions only in this restrictive setting. 

\begin{df}[Graph maps and train track (tt) maps] 
We call a continuous map $g \colon \Gamma \to \Gamma$ a \emph{graph map} if it takes $v$ to $v$ and is locally injective on edge interiors. Here we assume a graph map is also a homotopy equivalence. A graph map $g$ is a \emph{train track (tt) map} if additionally each $g^k$, with $k > 1$, is locally injective on edge interiors. 
\end{df}

In general we use definitions of \cite{bh92} and \cite{bfh00} for discussing train track maps. We remind the reader here of additional definitions and notation given in \cite{pI}.

\begin{con}[Edge sets] 
We denote by $\mE^+(\Gamma):= \{E_1, \dots, E_{n}\}$ the edge-set of $\Gamma$ with a prescribed orientation and set
$\mE(\Gamma):=\{E_1, \overline{E_1}, \dots, E_n, \ol{E_n} \}= \{e_1, e_2, \dots, e_{2n-1}, e_{2n} \}$, where $\ol{E_i}$ is $E_i$ oppositely oriented, $e_{2i}=E_i$ and $e_{2i+1}=\ol{E_i}$. If an indexing of $\mE^+(\Gamma)$ (thus of $\EG$) is prescribed, we call $\Gamma$ \emph{edge-indexed}.
\end{con} 

\begin{df}[Edge paths] 
Depending on the context, an \emph{edge-path} $e_1 \cdots e_n$ of \emph{length} $n$ will either mean a continuous map $[0,n] \to \Gam$ that, for each $1 \leq i \leq n$, maps $(i-1,i)$ homeomorphically to $int(e_i)$, or the sequence of edges $e_1, \dots, e_n$.
\end{df}

\begin{df}[Expanding irreducible tt maps] 
The tt map $g \colon \Gamma \to \Gamma$ is \emph{irreducible} if $g$ leaves invariant no proper subgraph with a noncontractible component. $g$ is \emph{expanding} if for each $E \in \mE^+(\Gamma)$, $length(f^n(E)) \to \infty$ as $n \to \infty$.
\end{df}

\begin{df}[Directions and turns] 
$\mD(\Gam)$ denotes the set of \emph{directions} at $v$, i.e. germs of edge segments emanating from $v$. As in the introduction, for each $e \in \mathcal{E}(\Gamma)$, we also let $e$ denote the initial direction of $e$. We let $Dg$ denote the direction map induced by $g$, i.e. $Dg(e)=e_1$ where $g(e)=e_1 \dots e_k$ for some $e_1, \dots, e_k \in \EG$. We call $d \in \mD(\Gamma)$ \emph{periodic} if $Dg^k(d)=d$ for some $k>0$ and \emph{fixed} if $k=1$.

By a \emph{turn} at $v$, we mean an unordered pair $\{d_i, d_j\}$ with $d_i, d_j \in \mD(\Gam)$. Then $g$ induces a turn map defined by $Dg(\{e_i, e_j\}) = \{Dg(e_i), Dg(e_j)\}$. For an edge-path $\gamma=e_1 \cdots e_k$ in $\Gamma$, we say $\gamma$ \emph{contains} or \emph{takes} $\{\overline{e_i}, e_{i+1} \}$ for each $1 \leq i < k$. We define
$$\mathcal{T}(g) := \bigcup_{\text{edges } \in \EG}\{\text{turns taken by } g(e)\}.$$
A turn $\{d_i, d_j\}$ is \emph{illegal} for $g$ if $Dg^p(d_i)=Dg^p(d_j)$ for some $p > 1$ and is considered \emph{legal} otherwise.
\end{df}

Given $\Phi \in Aut(F_r)$, defined for $X_1, \dots, X_r \in X$ by $\Phi(X_i)=x_{i,1} \cdots x_{i,n_i}$ for some $x_{i,k} \in X^{\pm 1}$, we define the graph map $g_{\Phi}$ corresponding to $\Phi$. Let $\Gamma = R_r$ with the identity marking. Then, for each $E_i \in \mE^+(\Gam)$, we have $g_{\Phi}(E_i)=e_{i,1} \cdots e_{i,n_i}$ with the $e_{i,k} \in \mE(\Gam)$ such that, under the identification of $F(X_1, \dots, X_r)$ with $\pi_1(R_r, v)$, we have that each $E_i$ and $e_{i,k}$ corresponds to $X_i$ and $x_{i,k}$.

By the correspondence of a $\Phi \in Aut(F_r)$ with a graph map $g_{\Phi}$, we can translate all of the language of directions, turns, etc, to free group automorphisms.

\subsection{Periodic Nielsen paths}{\label{ss:pnp}}

\begin{df}[Nielsen paths] Let $g \colon \Gamma \to \Gamma$ be an expanding irreducible train track map. Bestvina and Handel \cite{bh92} define a nontrivial tight path $\rho$ in $\Gamma$ to be a \emph{periodic Nielsen path (pNP)} if, for some power $R \geq 1$, we have $g^R(\rho) \cong \rho$ rel endpoints (and just a \emph{Nielsen path (NP)} if $R=1$). 

A NP $\rho$ is called \emph{indivisible} (hence is an ``iNP'') if it cannot be written as $\rho = \gamma_1\gamma_2$, where $\gamma_1$ and $\gamma_2$ are themselves NP's.
\end{df}

For a fully irreducible $\phi \in \mathcal{FI}_r$, \cite{bh92} gives an algorithm for finding an expanding irreducible tt representative (\emph{stable} representative) with the minimal number of iNP's.

\begin{df}[Rotationless]  An expanding irreducible tt map is called \emph{rotationless} if each periodic direction is fixed and each pNP is of period one. By \cite{fh11} Proposition 3.24, one can define a $\phi \in \mathcal{FI}_r$ to be \emph{rotationless} if one (hence all) of its tt representatives is rotationless.
\end{df}

Gaboriau, J\"aeger, Levitt, and Lustig introduce in \cite{gjll} the ``ageometric'' subclass of nongeometric fully irreducibles, defined there in terms of the index sum. We give here an equivalent definition in terms of periodic Nielsen paths.

\begin{df}[Ageometrics] $\phi \in \mathcal{FI}_r$ is called \emph{ageometric} if it has a tt representative with no pNP's (each stable representative of each rotationless power has no pNP's).
\end{df}

Since ageometrics are our main focus, we set notation and let $\ar$ denote the subset of $\mathcal{FI}_r$ consisting of the ageometric elements of $\mathcal{FI}_r$.

\subsection{Whitehead graphs}{\label{ss:iwg}}

Ideal Whitehead graphs were introduced by Handel and Mosher in \cite{hm11}. They are proved to be an outer automorphism invariant in \cite{pt}, where the equivalence of several ideal Whitehead graph definitions is also established. We give here an explanation in terms of a pNP-free tt representative $g \colon \Gamma \to \Gamma$ of a $\phi \in \mathcal{A}_r$, where $\Gamma$ is a rose with vertex $v$.

\begin{df}[Local Whitehead graphs $\mathcal{LW}$]
The \emph{local Whitehead graph} $\mathcal{LW}(g)$ for $g$ at $v$ has:
\begin{enumerate}
\item a vertex for each direction $d \in \mathcal{D}(\Gam)$ and
\item edges connecting vertices for $d_1, d_2 \in \mathcal{D}(\Gam)$ when $\{d_1, d_2\} \in \mT(g^p)$ for some power $p > 1$.
\end{enumerate}
\end{df}

\begin{df}[Ideal Whitehead graphs $\mathcal{IW}$]
The \emph{local stable Whitehead graph} $\swg$ is the subgraph of $\lwg$ obtained by restricting precisely to vertices with periodic direction labels and the edges of $\lwg$ connecting them. In this circumstance of a pNP-free representative on the rose, the \emph{ideal Whitehead graph $\iwp$ of $\phi$} is precisely $\swg$.
\end{df}

We need one more notion of a Whitehead graph. 
\begin{df}[Limited Whitehead graphs $\mW_L$]
Let $g \colon \Gam \to \Gam$ be a graph map on the rose. The \emph{limited Whitehead graph} ($\mW_L(g)$) for $g$ has:
\begin{enumerate}
\item a vertex for each direction $d \in \mD(\Gam)$ and
\item edges connecting vertices for $d_1, d_2 \in \mD(\Gam)$ when $\{d_1, d_2\} \in \mT(g)$.
\end{enumerate}
\end{df}

Notice that, if $g$ is a tt map, so that $\mathcal{LW}(g)$ is defined, $\mathcal{W}_L(g)$ is a subgraph of $\mathcal{LW}(g)$ and, in fact, $\mathcal{LW}(g) = \mathcal{W}_L(g^p)$ for some sufficiently high power $g^p$ of $g$.

\subsection{Ideal decompositions}{\label{ss:id}}

Let $r\ge 2$ and let $X$ be a free basis for $F_r$. By a \emph{standard Nielsen generator}, we mean an automorphism $\Phi \in \autr$ such that there exist $x,y\in X^{\pm 1}$ with $\Phi(x)=yx$ and $\Phi(z)=z$ for each $z\in X^{\pm 1}$ with $z\ne x^{\pm 1}$. We specify such $\Phi$ using notation $\Phi=[x\mapsto yx]$.

In \cite{pII} it is proved that, if the ideal Whitehead graph of $\phi \in \mathcal{A}_r$ is a connected $(2r-1)$-vertex graph, then there exists a rotationless power $\phi^R$ with a pNP-free tt representative on the rose that is a composition of graph maps corresponding to standard Nielsen generators:

\begin{prop}[{\cite[Proposition~3.3]{pI}}]{\label{P:ID}}
Let $\phi \in \mathcal{A}_r$ be such that $\mathcal{IW}(\phi)$ is a connected $(2r-1)$-vertex graph.
Then there exists a rotationless power $\psi=\phi^R$ and pNP-free tt representative $g$ of $\psi$ on the rose decomposing into the graph maps 
\begin{equation}{\label{e:id}}
\Gam=\Gamma_0 \xrightarrow{g_1} \Gamma_1 \xrightarrow{g_2} \cdots \xrightarrow{g_{n-1}} \Gamma_{n-1} \xrightarrow{g_n} \Gamma_n=\Gam
\end{equation}
\noindent where
{\begin{description}
\item [(I)] The index set $\{1, \dots, n \}$ is viewed as the set $\mathbf {Z}$/$n \mathbf {Z}$ with its natural cyclic ordering.
\newline
\item [(II)] Let $\Gamma_0$ be the base rose $R_r$ with the edges identified with the fixed generators of $F_r$. Each
$\Gamma_s$ is a rose with a marking $m_s \colon \Gamma_0 \to \Gamma_s$ so that, if we denote $m_s(e_t) = e_{s,t}$, for some $i_s$, $j_s$ with $e_{s,i_s} \neq (e_{s,j_s})^{\pm 1}$, we have 
\begin{equation}
g_s(e_{s-1,t}):=
      \begin{cases}
        e_{s,i_s} e_{s,j_s} \text{ for } t=j_s \\
        e_{s,t} \text{ for } e_t \neq e_{j_s}^{\pm 1}
      \end{cases}
\end{equation}
\item [(III)] For each $e_i \in \mathcal{E}(\Gamma)$ such that $i \neq j_n$, $Dg(e_i) = e_i$
\item [(IV)] $m_n$ is the identity map and $e_{n,t} = e_t$ for each $1 \leq t \leq 2r$.
\end{description}}
\end{prop}

As in \cite{pII}, we call tt maps satisfying (I)-(IV) of the proposition \emph{ideally decomposable (i.d.)} and call (\ref{e:id}) an \emph{ideal decomposition (i.d.)}.

\begin{nt}
We let $u \colon \{1, \dots, n \} \to \{1, \dots, 2r \}$ and $a \colon \{1, \dots, n \} \to \{1, \dots, 2r \}$ be defined by $u(s) = j_s$ and $a(s) = i_s$. Notice that the direction $e_{s, j_s}$ is missing from the image of $Dg_s$. We thus call it the \textbf{u}nachieved direction for $g_s$ (or in $\Gamma_s$) and denote it by $d^u_s$, i.e. $d^u_s=e_{s,u(s)}$. Since the direction $e_{s, i_s}$ is the image of two directions under $Dg_s$, we call it the twice-\textbf{a}chieved direction for $g_s$ (or in $\Gamma_s$) and denote it by $d^a_s$, i.e. $d^a_s=e_{s,a(s)}$.

We use the notation $f_k:= g_k \circ \cdots \circ g_1 \circ g_n \circ \cdots \circ g_{k+1}\colon \Gamma_k \to \Gamma_k$ and
\[g_{k,i}:=
\begin{cases}
    g_k \circ \cdots \circ g_i\colon \Gamma_{i-1} \to \Gamma_k \; \text{if} \; k>i\\
    g_k \circ \cdots \circ g_1 \circ g_n \circ \cdots \circ g_i \; \text{if} \; k<i\\
     \text{the identity } id \; \text{if} \; k=i-1. \\[-2mm]
\end{cases}
\]
\end{nt}

\begin{rk}{\label{R:Cyclic}}
It is proved in \cite{pI} that if $\Gamma = \Gamma_0 \xrightarrow{g_1} \Gamma_1 \xrightarrow{g_2} \cdots \xrightarrow{g_{n-1}}\Gamma_{n-1} \xrightarrow{g_n} \Gamma_n = \Gamma$ is an ideal decomposition of $g$, then $\Gamma_k \xrightarrow{g_{k+1}} \Gamma_{k+1} \xrightarrow{g_{k+2}} \cdots \xrightarrow{g_{k-1}} \Gamma_{k-1} \xrightarrow{g_k} \Gamma_k$ is an ideal decomposition of $f_k$ (and $f_k$ is in fact also a pNP-free tt representative of the same $\psi=\phi^R$ as $g$).
\end{rk}

\subsection{Lamination train track (ltt) structures}{\label{ss:ltt}}

Instead of just viewing the $g_i$ in an ideal decomposition as graph maps, we want to be able to ``blow up'' the vertices of the graphs $\Gamma_{i-1}$, $\Gamma_i$ to give them extra structure and then have $g_i$ induce a move between these structures. The structures are the ``ltt structures'' of \cite{pI} defined as follows (the moves are the extensions and switches of Subsection \ref{ss:se}):

\begin{df}[Lamination train track (ltt) structures $G(g)$]{\label{d:ltt}}

Let $g \colon \Gamma \to \Gamma$ be a pNP-free tt map on the rose (with vertex $v$). The \emph{colored local Whitehead graph} for $g$ at $v$, denoted $\cwg$, is the graph $\mathcal{LW}(g)$ with the subgraph $\mathcal{SW}(g)$ colored purple and $\mathcal{LW}(g)- \mathcal{SW}(g)$ colored red (nonperiodic direction vertices are red). Let $N(v)$ be a contractible open neighborhood of $v$ and $\Gamma_N=\Gamma-N(v)$. For each $E_i \in \mathcal{E}^+(\Gam)$, we add vertices labeled by the directions $E_i$ and $\overline{E_i}$ at the corresponding boundary points of the partial edge $E_i-(N(v) \cap E_i)$. The \emph{lamination train track (ltt) structure $G(g)$} for $g$ is formed from $\Gamma_N \bigsqcup \mathcal{CW}(g)$ by identifying each vertex $d_i$ in $\Gamma_N$ with the vertex $d_i$ in $\mathcal{CW}(g)$. Vertices for nonperiodic directions are red, edges of $\Gamma_N$ are black, and all periodic vertices are purple.

By the \emph{smooth structure} on $G(g)$ we mean the partition of the edges at each vertex into two sets: $\mathcal{E}_b$ (the black edges of $G(g)$) and $\mathcal{E}_c$ (the colored edges of $G(g)$). We call any path in $G(g)$ alternating between colored and black edges \emph{smooth}.
\end{df}

\begin{ex} For $g = g_{\Phi}$, where
$$\Phi =
\begin{cases}
a \mapsto abacbaba\bar{c}abacbaba \\
b \mapsto ba\bar{c} \\
c \mapsto c\bar{a}\bar{b}\bar{a}\bar{b}\bar{a}\bar{b}\bar{c}\bar{a}\bar{b}\bar{a}c
\end{cases}.
$$
the ltt structure looks like:
\begin{figure}[H]
\centering
\noindent \includegraphics[width=1.1in]{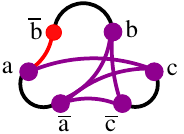}
\label{fig:lttExample} \\[-3mm]
\end{figure}
\end{ex}

In light of Remark \ref{R:Cyclic}, for an ideal decomposition $\Gamma = \Gamma_0 \xrightarrow{g_1} \Gamma_1 \xrightarrow{g_2} \cdots \xrightarrow{g_{n-1}}\Gamma_{n-1} \xrightarrow{g_n} \Gamma_n = \Gamma$, one in fact has an associated cyclic sequence of ltt structures $G_1=G(f_1), \dots, G_n=G(f_n)=G(g)$. For this reason we sometimes write an ideally decomposed map as $(g_1, \dots, g_n; G_1, \dots, G_n)$. 

\begin{rk}{\label{r:UniqueRed}}
Notice that the red vertex of each $G_i$ is labeled by $d^u_i$, which is the unique nonperiodic direction. Also, since $g_i$ (viewed as an automorphism) replaces each copy of $x_{u(i)}$ with $x_{a(i)}x_{u(i)}$, $G_i$ has a unique red edge and it connects $d^u_i$ to $\ol{d^a_i}$ (this argument is formulated more completely in \cite{pI}).
\end{rk}

The next subsection (Subsection \ref{ss:se}) is dedicated to describing the only two possible categories of moves relating the structures $G_i$ and $G_{i+1}$ in an ideal decomposition.

\subsection{Switches and extensions}{\label{ss:se}}

We need an abstract notion of an ltt structure so that we can take a potential ideal Whitehead graph, extend it to an abstract ltt structure, and see if this could be the ltt structure for a representative, as in Proposition \ref{P:ID}. In fact, we will want an entire sequence of ltt structures (and moves between them) to have an ideal decomposition.

Since we frequently deal with graphs whose vertices are labeled by the directions at the vertex of a rose (sometimes abstractly), we first establish notation for discussing such graphs and their vertex-labeling sets. We call a $2r$-element set of the form $\{X_1, \overline{X_1}, \dots, X_r, \overline{X_r} \}$, with elements paired into \emph{edge pairs} $\{X_i, \overline{X_i}\}$, a \emph{rank}-$r$ \emph{edge pair labeling set} (we write $\overline{\overline{X_i}}=X_i$). We call a graph with vertices labeled by an edge pair labeling set a \emph{pair-labeled} graph, and an \emph{indexed pair-labeled} graph if an indexing is prescribed.

\begin{df}{\label{d:abstractltt}}[(Abstract) lamination train track (ltt) structures]
A \emph{lamination train track (ltt) structure} is a colored pair-labeled graph $G$ (black edges are included, but not considered colored) satisfying
\noindent {\begin{description}
\item [I] Each vertex has valence at least $2$.
\item [II] Each edge has $2$ distinct vertices.
\item [III] Vertices are either purple or red.
\item [IV] Edges are of $3$ types:
{\begin{description}
\item[Black Edges] $G$ has a single black edge connecting each pair of (edge-pair)-labeled vertices. There are no other black edges. Thus, each vertex is contained in a unique black edge.
\item[Red Edges] A colored edge is red if and only if at least one of its endpoint vertices is red.
\item[Purple Edges] A colored edge is purple if and only if both endpoint vertices are purple.
\end{description}}
\item [V] Two distinct colored edges never connect the same pair of vertices.
\end{description}}
It is proved in \cite{pI} that, in our situation, $G(g)$ has a unique red vertex and red edge (see Remark \ref{r:UniqueRed}). Hence, we also require:
\noindent {\begin{description}
\item [VI] $G$ has precisely $2r-1$ purple vertices, a unique red vertex, and a unique red edge.
\end{description}}
\end{df}

We consider ltt structures \emph{equivalent} that differ by an ornamentation-preserving graph isomorphism (homeomorphism taking vertices to vertices and edges to edges and preserving colors and labels).

We say that an ltt structure is \emph{based at} a rose $\Gam$ if its vertex labelling set is $\mD(\Gam)$.

\begin{nt}{\label{n:ltt}}
We let $[x,y]$ denote the edge connecting the vertex pair labeled by $\{x,y\}$. To be consistent with the situation of Definition \ref{d:ltt}, we call the unique red vertex $d^u$. Also, we call the unique red edge $[d^u, \ol{d^a}]$. In particular, $\ol{d^a}$ labels the purple vertex of the red edge.

Suppose one has a sequence of ltt structures $G_k, \dots, G_n$ such that, for each $k \leq i \leq n$, we have that $G_i$ is based at the rose $\Gam_i$. Then the notation will carry indices (for example, $d^u_i$ will denote the red vertex in $G_i$). Additionally, $e_{u(i)}$ denotes the edge in the $\Gam_i$ whose initial direction is $d^u_i$ and $e_{a(i)}$ denotes the edge in $\Gam_i$ whose terminal direction is $\ol{d^a_i}$.

Since we are looking for an ltt structure $G$ which could belong to a representative with some potential ideal Whitehead graph $\mG$ as its ideal Whitehead graph (in which case $\mG$ would be the purple subgraph of $G$), we denote the purple subgraph by \emph{$\mP(G)$} and, if $\mG \cong \mP(G)$, say $G$ is an \emph{ltt structure for $\mG$}. For the same reason, in light of the \cite{pI} Birecurrency Condition (of which we remind the reader in Proposition \ref{p:BC} below), we call an ltt structure \emph{admissible} which is \emph{birecurrent}, i.e has a locally smoothly embedded line traversing each edge infinitely many times as $\RR \to \infty$ and as $\RR \to -\infty$. $\mcg$ will denote the colored subgraph of $G$.

Given a graph map $g_k \colon \Gam_{k-1} \to \Gam_k$ as in Proposition \ref{P:ID}(II) and ltt structures $G_i$ based at $\Gam_i$, for $i = k-1, k$, the induced map $D^Cg_k \colon \mC(G_{k-1}) \to \mC(G_k)$ (when it exists) is defined by sending each vertex $d$ in $G_{k-1}$ to the vertex $Dg_k(d)$ in $G_k$ and each edge $[d_1, d_2]$ in $\mC(G_{k-1}$ to $[Dg_k(d_1), Dg_k(d_2)]$.
\end{nt}

\begin{df}[Generating triples]
By a \emph{generating triple} $(g_k; G_{k-1}, G_k)$ for $\mG$ we mean an ordered set of three objects, where 
\begin{description}
\item [gtI] $g_k \colon \Gam_{k-1} \to \Gam_k$ is a graph map as in Proposition \ref{P:ID}(II).
\item [gtII] For $i = k-1, k$, we have that $G_i$ is an ltt structure for $\mG$ based at $\Gam_i$ and in fact:
{\begin{itemize}
\item the red vertex of $G_k$ is labelled by $e_{u(k)}$, i.e. $d^u_k = e_{u(k)}$, and
\item the red edge of $G_k$ is $[e_{u(k)}, \ol{e_{a(k)}}]$, i.e. $d^a_k = e_{a(k)}$.
\end{itemize}}
\item [gtIII] The induced map $D^Cg_k \colon \mC(G_{k-1}) \to \mC(G_k)$ exists and restricts to a graph isomorphism from $\mP(G_{k-1})$ to $\mP(G_k)$.
\end{description}
\end{df}

The triple is \emph{admissible} if
\begin{enumerate}
\item each $G_i$ is admissible and
\item either $u(k-1) = u(k)$ or $u(k-1) = a(k)$.
\end{enumerate}

\begin{rk}
The inspiration for (2) is its necessity in an ideal decomposition for ensuring that the composition is a tt map with $2r-1$ periodic directions.
\end{rk}

\vskip5pt

Given an ltt structure $G_k$ for $\mG$ and a \emph{determining} purple edge $[d^a_k,d_{k,l}]$, there are potentially two admissible triples $(g_k; G_{k-1}, G_k)$ that could arise in an ideal decomposition. They correspond to the situation where $u(k-1) = u(k)$ (the ``extension'' situation) and the situation where $u(k-1) = a(k)$ (the ``switch'' situation).

\begin{df}[Extensions]
The \emph{extension} determined by $[d^a_k,d_{k,l}]$ is the generating triple $(g_k; G_{k-1}, G_k)$ for $\mG$ satisfying
\begin{description}
\item [extI] $u(k-1) = u(k)$, thus the restriction of $D^Cg_k$ to $\mP(G_{k-1})$ is an isomorphism sending the vertex $e_{k-1,i}$ to the vertex $e_{k,i}$ for each $i \neq u(k-1)$.
\item [extII] $e_{k-1, u(k)} = d^u_{k-1}$, i.e. is the red vertex of $G_{k-1}$.
\item [extIII] $\ol{d^a_{k-1}}=d_{k-1,l}$, i.e. the purple vertex of the red edge in $G_{k-1}$ is $d_{k-1,l}$.
\end{description}
\end{df} 

\begin{df}[Switches]
The \emph{switch} determined by $[d^a_k,d_{k,l}]$ is the generating triple $(g_k; G_{k-1}, G_k)$ for $\mG$ satisfying
\begin{description}
\item [swI] $u(k-1) = a(k)$, thus $D^Cg_k$ restricts to an isomorphism from $\mP(G_{k-1})$ to $\mP(G_{k})$ defined by
\[
\begin{cases}
    e_{k-1, u(k)} \mapsto e_{k, a(k)} = e_{k, u(k-1)}\\
    e_{k-1, s} \mapsto e_{k, s} \text{ for } s \neq u(k)\\
\end{cases}
\]
\item [swII] $e_{k-1, a(k)} = d^u_{k-1}$, i.e. is the red vertex of $G_{k-1}$.
\item [swIII] $\ol{d^a_{k-1}}=d_{k-1,l}$, i.e. the purple vertex of the red edge in $G_{k-1}$ is $d_{k-1,l}$.
\end{description}
\end{df} 

\begin{df}[Admissible compositions]
An \emph{admissible composition} $(g_{i-k}, \dots, g_i; G_{i-k-1}, \dots, G_i)$ for $\mG$ with $0 \leq k <i$ consists of
\begin{itemize}
\item a sequence of automorphisms $g_{i-k}, \dots, g_i$ such that $\Gamma_{i-k-1} \xrightarrow{g_{i-k}} \cdots  \xrightarrow{g_i} \Gamma_i$ satisfies Proposition \ref{P:ID} (I)-(III) and
\item a sequence of ltt structures $G_{i-k-1}, \dots, G_i$ for $\mG$ so that each $(g_j; G_{j-1}, G_j)$, with $i-k \leq j \leq i$ is either an admissible switch or an admissible extension.
\end{itemize}
When $k=1$, we call the composition an \emph{i.d. admissible triple}. We may also write that $(g_{i-k-1}, \dots, g_i; G_{i-k-1}, \dots, G_i)$ is an admissible composition if 
$$g_{i-k-1}=[e_{i-k-2,u(i-k-1)} \mapsto e_{i-k-1,a(i-k-1)}e_{i-k-1,u(i-k-1)}] \colon \Gam_{i-k-2} \to \Gam_{i-k-1}$$ 
is additionally as in Proposition \ref{P:ID}, $d^u_{i-k-1}=e_{i-k-1,u(i-k-1)}$, and $d^a_{i-k-1}=e_{i-k-1,a(i-k-1)}$.

Two admissible compositions $(g_{i-k}, \dots, g_i; G_{i-k-1}, \dots, G_i)$ and $(g'_{i-k}, \dots, g'_i; G'_{i-k-1}, \dots, G'_i)$ for the same $\mG$ are considered \emph{equivalent} if, for each $i-k-1 \leq j \leq i$, we have that $G_j$ is equivalent to $G_j'$ and that $g_j$ and $g_j'$ correspond to the same standard Nielsen generator.
\end{df}

\subsection{Ideal decomposition diagrams ($\mathcal{ID}(\mathcal{G})$)}{\label{ss:idd}}

Recall that a directed graph is strongly connected, if for each pair of vertices $v_1$, $v_2$ in the graph, the graph contains a directed path from $v_1$ to $v_2$. Notice that one can find the union of the maximal strongly connected components in a graph by taking the union of all of the directed loops in the graph. Since we want a diagram containing a loop for each ideally decomposed pNP-free representative with a given ideal Whitehead graph, this is precisely what we want. Such a diagram is defined in \cite{pI}:

\begin{df}[Ideal decomposition diagrams] Let $\mG$ be a connected $(2r-1)$-vertex graph. The \emph{ideal decomposition diagram} for $\mG$ (or $\idg$) is defined to be the disjoint union of the maximal strongly connected subgraphs of the directed graph where:
\begin{description}
\item[Nodes] The nodes are equivalence classes of admissible indexed ltt structures for $\mG$.
\item[Edges] For each equivalence class of an i.d. admissible triple $(g_i; G_{i-1}, G_i)$ for $\mG$, there is a directed edge $E(g_i; G_{i-1}, G_i)$ in $\idg$ from the node $[G_{i-1}]$ to the node $[G_i]$.
\end{description}
\end{df}

\cite{pt} gives a procedure for constructing $\id$ diagrams (there called ``$\mathcal{AM}$ diagrams'').

\section{Unachieved Graphs}{\label{S:Unachieved}}

The three connected, $(2r-1)$-vertex graphs that are not the ideal Whitehead graph for an outer automorphism in $\mathcal{A}_3$ are proved not to be achieved as such in \cite{pI}. We include here the tools used to prove they are not achieved, as they can be viewed as preliminary checks performed on a graph $\mG$ before applying the strategies described below to find a representative yielding $\mG$.

\begin{prop}[Birecurrency Condition]{\label{p:BC}}\cite{pI} 
Let $\phi \in \mathcal{A}_r$ be such that $\mathcal{IW}(\phi)$ is connected with $2r-1$ vertices. Then the ltt structure $G(g)$ for each pNP-free tt representative $g$ of each power $\phi^p$, with $p \geq 1$, is birecurrent. 
\end{prop}

Proposition \ref{p:BC} can in fact be used to prove that Graph II is unachieved, as it is shown in \cite{pI} that there is no admissible (birecurrent, in particular) ltt structure for Graph II. It also explains why the ltt structures giving the nodes of an $\id$ diagram must be birecurrent.

The following test is inspired by our need for our representatives to be irreducible:

\vskip5pt

\noindent \textbf{Irreducibility Potential Test \cite{pI}:} Check whether, in each connected component of $\mathcal{ID}(\mG)$, for each edge vertex pair $\{d_i, \overline{d_i}\}$, there is a node $N$ in the component such that either $d_i$ or $\overline{d_i}$ labels the red vertex in the structure $N$. If it holds for no component, then $\mG$ is unachieved.

\vskip5pt

It is shown in \cite{pI} that no component of the $\id$ diagram for Graph V or for Graph VII passes the test, which is how we show that neither is ``achieved.''

\section{Representative construction strategies}{\label{s:Procedure}}

\vskip10pt

\subsection{Underlying Strategy}{\label{S:UnderlyingStrategy}}

The following proposition (\cite{pI}, Proposition 8.3) allows us to search for representatives realized as loops in $\mathcal{ID}$ diagrams:

\begin{prop}[\cite{pI}, Proposition 8.3]{\label{P:ReferenceLoop}} Suppose $(g_1, \dots, g_n; G_0, G_1 \dots, G_{n-1}, G_n)$ is a pNP-free rotationless i.d. tt representative of a $\phi \in \mathcal{A}_r$ such that $\mathcal{IW}(\phi)=\mathcal{G}$ is a connected $(2r-1)$-vertex graph. Then $E(g_1; G_0, G_1) * \dots * E(g_n; G_{n-1}, G_n)$ exists in $\idg$ and forms an oriented loop. \end{prop}

Inspired by the Full Irreducibility Criterion of \cite{pII}, the following lemma (\cite{pII} Lemma 4.2) tells us what kind of loops to look for:

\begin{lem}[\cite{pII}, Lemma 4.2]{\label{l:RepresentativeLoops}}
Suppose $\mG$ is a connected $(2r-1)$-vertex graph, $g = g_n \circ \cdots \circ g_1$ is rotationless, and 
$$L(g_1, \dots, g_n; G_0, G_1 \dots, G_{n-1}, G_n)= E(g_1; G_{0}, G_1) * \dots * E(g_n; G_{n-1}, G_n)$$ 
is a loop in $\mathcal{ID}(\mG)$ satisfying each of the following:
\begin{description}
\item [A] Each purple edge of $G(g_{n,1})$ is labelled by a turn taken by some $g^p(e_j)$, where $p \geq 1$ and $e_j \in \mathcal{E}(\Gamma_0)$.
\item [B] For each $1 \leq i,j \leq r$, there exists some $p \geq 1$ such that $g^p(E_j)$ contains either $E_i$ or $\overline{E_i}$.
\item [C] $g$ has no periodic Nielsen paths. \\[-6mm]
\end{description}
\noindent Then $g$ is a tt representative of some $\phi \in \mathcal{A}_r$ such that $\iw(\phi)=\mG$.
\end{lem}

In light of Proposition \ref{P:ReferenceLoop}, given a potential ideal Whitehead graph $\mG$ (connected, with $2r-1$ vertices), our strategies involve finding a loop in $\idg$ for a rotationless pNP-free tt representative $g$ of $\phi \in \ar$ with $\iw(\phi) \cong \mG$. In light of Lemma \ref{l:RepresentativeLoops}, our main obstacles are ensuring that:
~\\
\vspace{-5mm}
\begin{description}
\item [st1] The loop is indeed a loop (returns to the initial ltt structure).
\item [st2] The entire graph $\mG$ is obtained (as in Lemma \ref{l:RepresentativeLoops}A).
\item [st3] $g$ is irreducible (as in Lemma \ref{l:RepresentativeLoops}B).
\item [st4] $g$ has no pNP's (as in Lemma \ref{l:RepresentativeLoops}C).
\item [st5] $g$ is rotationless.
\end{description}

\begin{rk}
It will also be important in both strategies presented below that the map obtained in fact lives in $\idg$. For this reason one must check the admissibility of the triples, in particular, the birecurrence of each of the ltt structures.
\end{rk}

\begin{rk}{\label{r:Checks}}
Lemma \ref{l:RepresentativeLoops} tells us how to address (st3). One can check for (st5) by composing generator direction maps (if $g = g_n \circ \cdots \circ g_1$, then $Dg = Dg_n \circ \cdots \circ Dg_1$). This obstacle is then simple to address by taking an adequate power to fix the periodic directions. One can check for the existence of pNP's (by hand, as in \cite{pII}, or using the computer package of Thierry Coulbois, for example). While this was not necessary for constructing any of the maps appearing in the proof of Theorem \ref{T:MainTheorem}, one could construct a ``Nielsen path prevention sequence'' (as in \cite{pII}) to ensure that there are no pNP's. 
\end{rk}

One of the trickiest obstacles to address is (st2). Hence, we devote Subsection \ref{S:TrackingProgress} precisely to addressing this obstacle.

\vskip10pt

\subsection{Tracking progress and construction compositions}{\label{S:TrackingProgress}} 

This subsection is devoted to explaining how we make sure an entire ideal Whitehead graph is actually achieved. An important lemma we include first is indirectly proved in \cite{pI}. We provide a direct proof here for completeness.

\begin{lem}{\label{L:LimitedWGs}}
Suppose $(g_m, \dots, g_n; G_m, \dots, G_n)$ is an admissible composition with the standard notation that $g_i =  [e_{i-1,u(i)} \to e_{i,a(i)} e_{i,u(i)}]$, for each $m \leq i \leq n$. Then 
\begin{description}
\item [A] for each $E_{m-1,t} \in \mE^+(\Gam_{m-1})$, we have that $g(E_{m-1,t})$ contains $E_{n,t}$,
\item [B] $g$ is a graph map,
\item [C] $$\mT(g_{n,m}) = \bigcup_{k = m}^{n}Dg_{n,k+1}(\{\ol{e_{k,a(k)}}, e_{k,u(k)}\}),$$
\item [D] and for each $m \leq s \leq n$, $$\mT(g_{n,m}) = Dg_{n,s+1}(\mT(g_{s,m})) \cup \mT(g_{n,s+1}).$$
\end{description}
\end{lem}

\begin{proof}
(D) follows directly from (C). We proceed by induction on $n-m$ to prove (A)-(C). For the base case suppose $n-m=0$, so that $g_{n,m}=g_m$. For $e_{m-1,i} \neq e_{m-1,u(m)}^{\pm 1}$, we have that $g_m(e_{m-1,i})=e_{m,i}$, which clearly has no back-tracking, contains $e_{m,i}$, and provides no taken turns. $g_m(e_{m-1,u(m)})=e_{m,a(m)}e_{m,u(m)}$, which also has no back-tracking (since each $g_m$ corresponds to a standard Nielsen generator), contains $e_{m,u(m)}$, and contains precisely the turn $\{\ol{e_{m,a(m)}}, e_{m,u(m)}\}$. Also, $g_m(\ol{e_{m-1,u(m)}})=\ol{e_{m,u(m)}}$ $\ol{e_{m,a(m)}}$, which again has no back-tracking (since it is the path $e_{m,a(m)}e_{m,u(m)}$ oppositely oriented), contains $\ol{e_{m,u(m)}}$, and contains precisely the same turn $\{\ol{e_{m,a(m)}}, e_{m,u(m)}\}$ as $g_m(e_{m-1,u(m)})$. 

Now suppose $n-m > 1$ and (A)-(C) hold for all $n \geq n' \geq m' \geq m$ with $n'-m' \leq n-m$. In particular, $g_{n-1,m}(E_{m-1,t})$ contains $E_{n-1,t}$ for each $E_{m-1,t} \in \mE^+(\Gam_{m-1})$, $g_{n-1,m}$ is a graph map, and 
$$\mT(g_{n-1,m}) = \bigcup_{k = m}^{n-1}Dg_{n-1,k+1}(\{\ol{e_{k,a(k)}}, e_{k,u(k)}\}).$$

For each $E_{m-1,t} \in \mE^+(\Gam_{m-1})$, there is an $N_j > 0$ and function $J \colon \{1, \dots, N_j\} \to \NN_{>0}$ so that $g_{n-1,m}(E_{m-1,j}) = e_{n-1,J(1)} \cdots e_{n-1,J(N_j)}$ for some $e_{n-1,J(t)} \in \mE(\Gam_{n-1})$. We first show that no $g_{n,m}(E_{m-1,j})$ has cancellation. Since $g_{n-1,m}(E_{m-1,j})$ has no cancellation and $g_n$ is also a graph map, it suffices to show that we never have $Dg_n(\ol{e_{n-1,J(i)}}) = Dg_n(e_{n-1,J(i+1)})$. In other words, it suffices to show that no turn in $\mT(g_{n-1,m})$ is the illegal turn $\{e_{n-1,u(n)}, e_{n-1,a(n)}\}$ for $g_n$. Now
$$\mT(g_{n-1,m}) = \{\ol{e_{n-1,a(n-1)}}, e_{n-1,u(n-1)}\} \cup Dg_{n-1}(\bigcup_{k = m}^{n-2}Dg_{n-2,k+1}(\{\ol{e_{k,a(k)}}, e_{k,u(k)}\})).$$
Since the drection $e_{n-1,u(n)}$ is not in the image of $Dg_n$, we are left to show that we cannot have
\begin{equation}{\label{e:EqualTurns}}
\{\ol{e_{n-1,a(n-1)}}, e_{n-1,u(n-1)}\} = \{e_{n-1,u(n)}, e_{n-1,a(n)}\}.
\end{equation}
Since $(g_n; G_{n-1}, G_n)$ is either an admissible switch or extension, either $u(n-1) = a(n)$ or $u(n-1) = u(n)$. We consider separately the switch ($u(n-1) = a(n)$) and extension ($u(n-1) = u(n)$) cases. 

If $u(n-1) = u(n)$, then (\ref{e:EqualTurns}) becomes 
\begin{equation}{\label{e:NewEqualTurns}}
\{\ol{e_{n-1,a(n-1)}}, e_{n-1,u(n)}\} = \{e_{n-1,a(n)}, e_{n-1,u(n)}\},
\end{equation}
\noindent so equality in (\ref{e:EqualTurns}) would imply $\ol{e_{n-1,a(n-1)}} = e_{n-1,a(n)}$. But then the ``determining'' edge for the extension would be $[e_{n,a(n)},e_{n,a(n)}]$, which cannot be a purple edge in an ltt structure.

If $u(n-1) = a(n)$, then (\ref{e:EqualTurns}) becomes $\{\ol{e_{n-1,a(n-1)}}, e_{n-1,a(n)}\} = \{e_{n-1,u(n)}, e_{n-1,a(n)}\}$, so equality would imply $\ol{e_{n-1,a(n-1)}} = e_{n-1,u(n)}$. But then the ``determining'' edge for the switch would be $[e_{n,a(n)},e_{n,u(n)}]$, which cannot be a purple edge in an ltt structure,as $e_{n,u(n)}$ is red. So it suffices to show that $\{e_{n-1,u(n)}, e_{n-1,a(n)}\}$ is not in the image of $Dg_{n-1}$. Since $g_{n-1}=[e_{n-2,u(n-1)} \mapsto e_{n-1,a(n-1)}e_{n-1,u(n-1)}]$, the image of $Dg_{n-1}$ is missing precisely the direction $e_{n-1,u(n-1)}$. Hence, $\{e_{n-1,u(n)}, e_{n-1,a(n)}\}$ is not in the image of the induced turn map and no $g_{n,m}(E_{m-1,j})$ has cancellation and $g_{n,m}$ is a graph map.

Now, for each $E_{m-1,j}$, we have 
$$g_{n,m}(E_{m-1,j})=g_n(e_{n-1,J(1)} \cdots e_{n-1,J(N_j)})=g_n(e_{n-1,J(1)}) \cdots g_n(e_{n-1,J(N_j)}).$$
By the inductive hypotheses, $E_{n-1,j}$ is contained in $g_{n-1,m}(E_{m-1,j})$. Thus, some $e_{n-1,J(i)}=E_{n-1,j}$. And $g_n(e_{n-1,J(i)})$ contains $E_{n,j}$ by the inductive hypotheses. So we are left to prove (C).

$\mT(g_{n,m})$ consists precisely of the turns taken by the $g_n(e_{n-1,J(i)})$ and the turns 
\begin{equation}{\label{e:TakenTurns}}
\{Dg_n(\ol{e_{n-1,J(i)}}), Dg_n(e_{n-1,J(i+1)})\}.
\end{equation}
The set of turns of the form of (\ref{e:TakenTurns}) is precisely $Dg_n(\mT(g_{n-1,m}))$.

Now notice that 
$$\bigcup_{k = m}^{n}Dg_{n,k+1}(\{\ol{e_{k,a(k)}}, e_{k,u(k)}\})=
\{\ol{e_{n,a(n)}}, e_{n,u(n)}\}\cup Dg_n(\mT(g_{n-1,m})).$$
So we are left to show that the set of turns taken by the $g_n(e_{n-1,J(i)})$ is precisely $\{\ol{e_{n,a(n)}}, e_{n,u(n)}\}$. This follows from the fact that $g_n=[e_{n-1,u(n)} \mapsto e_{n,a(n)}e_{n,u(n)}]$ and that each $g_{n-1,m}(E_{m-1,j})$ contains $E_{n-1,m}$.
\qedhere
\end{proof}

Before presenting the two classes of strategies, we show how to check Lemma \ref{l:RepresentativeLoops}A. In light of Lemma \ref{L:LimitedWGs}, one can check that the entire graph is built by taking images of the red edges created by $g_i$, as in the following example:

\vskip4pt

\begin{ex}{\label{E:GraphBuilding}} In light of Lemma \ref{L:LimitedWGs}, we show here an example of how to check that all of $\mathcal{G}$ is ``built'' (we iteratively take the image under each $Dg_k$ of the edges ``created'' thus far):
\begin{figure}[H]
\centering
\noindent \includegraphics[width=4.7in]{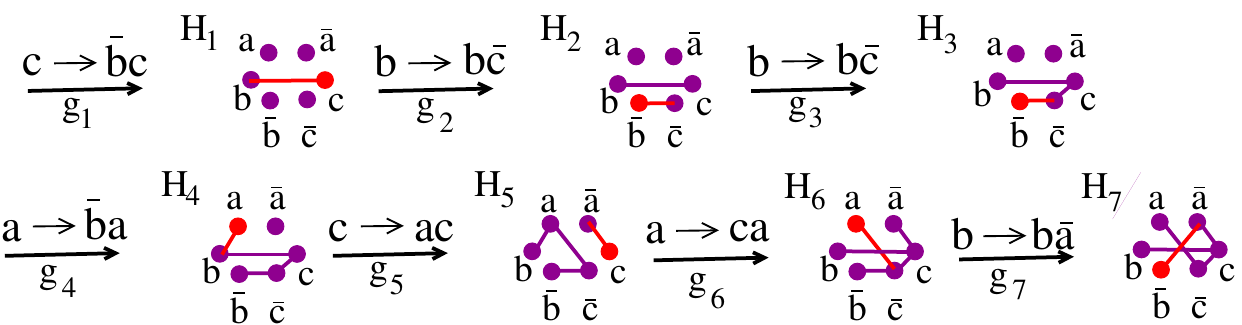}
\label{fig:TrackingProgress}
\end{figure}

We include subgraphs $H_i=\mW_L(g_{i,1})$ of the ltt structures $G_i$ to track how edges are ``built.'' $g_1=[c \mapsto \bar{b}c]$.  Thus, the red edge of $H_1$ is $[c,b]$, with red vertex $c$. Since $g_2=[b \mapsto b\bar{c}]$, the red edge in $H_2$ is $[\bar{b},\bar{c}]$, with red vertex $\bar{c}$. $H_2$ will also contain the image $[c,b]$ of the red edge $[c,b]$ under $Dg_2:\bar{b} \mapsto c$. Since $g_3=[b \mapsto b\bar{c}$], the red edge in $H_3$ will be $[\bar{b},\bar{c}]$, with red vertex $\bar{b}$. $H_3$ will also contain the image $[c,\bar{c}]$ of the red edge $[\bar{b},\bar{c}]$ and the image $[c,b]$ of the purple edge $[c,b]$ under $Dg_3:\bar{b} \mapsto c$. Since $g_4=[a \mapsto \bar{b}a]$, the red edge in $H_4$ will be $[a,b]$, with red vertex $a$. $H_4$ will also contain the image $[\bar{b},\bar{c}]$ of the red edge $[\bar{b},\bar{c}]$ and the images $[c,b]$ and $[c,\bar{c}]$ of the purple edges $[c,b]$ and $[c,\bar{c}]$ under $Dg_4:a \mapsto \bar{b}$. The remaining $H_i$ are constructed similarly.
\end{ex}

\subsubsection{Construction paths}{\label{sss:constructions}}

We present a key tool used to ensure that all edges of a potential ideal Whitehead graph appear.

\begin{df}
We call a composition of standard generators $\Phi = \Phi_n \circ \cdots \circ \Phi_1$, with $\Phi_i \colon x_i \to y_ix_i$ for each $i$, a \emph{construction automorphism} if $x_i = x_j$ for each $1 \leq i,j \leq n$. 
\end{df}

\begin{ex}{\label{E:ConstructionAutomorphism}}

Consider the graph maps induced by a construction automorphism $g = g_i \circ \cdots \circ g_{i-4}$ where $g_{i-4} =  [a \to a \bar c]$, $g_{i-3} =  [a \to ab]$, $g_{i-2} =  [a \to a \bar c]$, $g_{i-1} =  [a \to a \bar b$], and $g_i = [a \to a \bar c]$. The following depicts the limited Whitehead graphs $W_j = \mathcal{W}_L(g_j \circ \cdots \circ g_{i-4})$.
~\\
\vspace{-5mm}
\begin{figure}[H]
\centering
\noindent \includegraphics[width=5in]{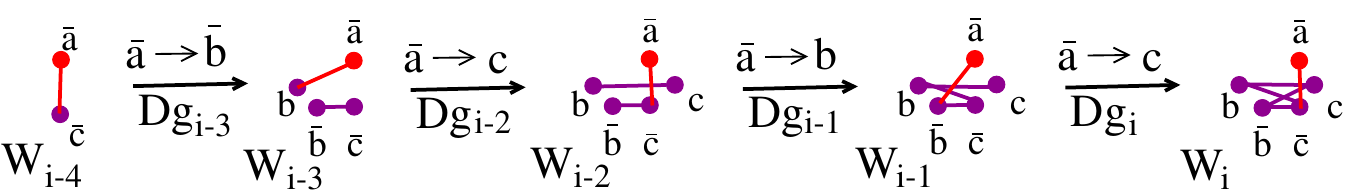} \\[-2mm]
\label{fig:ConstructingExample}
\end{figure}

It can be noted that the edges of $W_i$ are the colored edges of a smooth path (in the ltt structure sense) in the graph depicted below (obtained by adding the black edges $[b, \bar b]$ and $[c, \bar c]$). The edges are traversed in the reverse order of their ``addition.''
~\\
\vspace{-7mm}
\begin{figure}[H]
\centering
\noindent \includegraphics[width=1in]{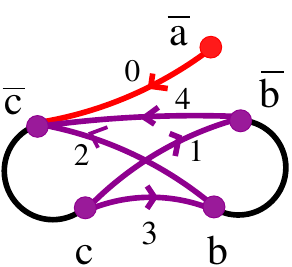} \\[-3.5mm]
\label{fig:ConstructionPath}
\end{figure}

\end{ex}

The ``construction'' of such a path by a construction automorphism is proved to always occur in \cite{pII} (see Lemma \ref{L:Building} below).

In building automorphisms yielding particular ideal Whitehead graphs, we reverse this construction procedure. We start with a smooth path (a ``construction path'') in a certain subgraph (the ``construction subgraph'') of a potential ltt structure and then use the path to build an automorphism, in fact a \emph{construction composition}: 

\begin{df}
An \emph{admissible construction composition} for a connected, $(2r-1)$-vertex graph $\mathcal{G}$ is an admissible composition $(g_{i-k}, \dots, g_i; G_{i-k-1}, \dots, G_i)$ for $\mathcal{G}$ such that
\indent{\begin{description}
\item  [(CC1)] each  $(g_j, G_{j-1}, G_j)$ with $i-k < j \leq i$ is an extension for $\mathcal{G}$ and
\item  [(CC2)] $(g_{i-k}, G_{i-k-1}, G_{i-k})$ is a switch for $\mathcal{G}$.
\end{description}}
We call the composition without the initial switch a \emph{pure construction composition}.
\end{df}

Below we include the abstract definitions necessary to make the processes of Example \ref{E:ConstructionAutomorphism} formal. However, first, since we use this procedure repeatedly in what follows, we explain how to find the construction automorphism (and intermediary ltt structures) from a construction path.

\begin{ex}{\label{E:AutomorphismFromPath}}

We start on the right with an ltt structure $G_i$ including the potential construction path (of Example \ref{E:ConstructionAutomorphism}). Since we want each triple to be an extension, we keep the purple graph the same in each $G_j$. In the ltt structure $G_{i-k}$, we attach $\bar a$ by a red edge to the terminal vertex of the edge labeled in $G_i$ by $k$. We determine $g_{i-k}$ by the red edge of $G_{i-k}$. If the red edge is $[e_{u(i-k)}, \ol{e_{a(i-k)}}]$ then, to make $(g_{i-k}; G_{i-k-1}, G_{i-k})$ a generating triple, we set $g_{i-k}=[e_{u(i-k)} \to e_{a(i-k)}e_{u(i-k)}]$.

\begin{figure}[H]
\centering
\noindent \includegraphics[width=6.5in]{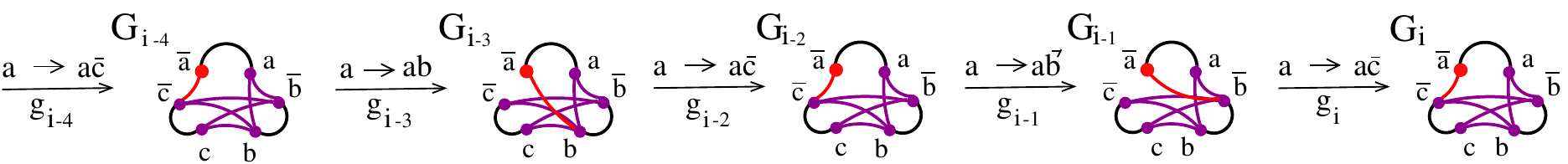} 
\label{fig:ConstructionPathLtt}
\end{figure}

\end{ex}

\begin{df}
The \emph{construction subgraph} $G_{con}$ is constructed from $G$ via the following procedure:
\begin{itemize}
\item [1.] Remove the interior of the black edge $[d^u, \overline{d^u}]$, the purple vertex $\overline{d^u}$, and the interior of any purple edges containing $\overline{d^u}$. Call the graph with these edges and vertices removed $G^1$.
\item [2.] Given $G^{j-1}$, recursively define $G^j$: Let $\{\alpha_{j-1,i} \}$ be the set of vertices in $G^{j-1}$ not contained in any colored edge of $G^{j-1}$. $G^j$ is obtained from $G^{j-1}$ by removing all black edges containing a vertex $\alpha_{j-1,i} \in \{\alpha_{j-1,i} \}$, as well as the interior of each purple edge containing a vertex $\overline{\alpha_{j-1,i}}$.
\item [3.] $G_{con} = \underset{j}{\cap} G^j$.
\end{itemize}
\end{df}

Because \cite{pII} Lemma 3.6 plays such a crucial role in ensuring (st2), we restate the lemma in the form in which we use it (still implied by its proof, combined with \cite{pI} Lemma 5.7). It gives some kind of formalization of Examples \ref{E:ConstructionAutomorphism} and \ref{E:AutomorphismFromPath}.

\begin{lem}[\cite{pII} Lemma 3.2, Lemma 3.6]{\label{L:Building}}
Let $(g_1, \dots, g_n; G_1, \dots, G_n)$ be an i.d. representative of a $\phar$ with $\iwp \cong \mG$, where $\mG$ is a connected, $(2r-1)$-vertex graph. Suppose that some $(g_{i-k}, \dots, g_i; G_{i-k-1}, \dots, G_i)$, with $1 \leq i-k \leq i \leq n$, is a construction composition. Then:
\begin{enumerate}
\item The sequence of vertices $e_{i,u(i)}, \ol{e_{i,a(i)}}, e_{i,a(i)}, \ol{e_{i,a(i-1)}}, \dots, \ol{e_{i,a(i-k)}}, e_{i,a(i-k)}$ defines a smooth path in $G_i$ (starting with the red edge $[e_{i,u(i)}, \ol{e_{i,a(i)}}]$ oriented from its red vertex to its purple vertex). We call such a path the \emph{construction path} for $(g_{i-k}, \dots, g_i; G_{i-k-1}, \dots, G_i)$.
\item Letting $H$ denote the subgraph of $G_i$ consisting of the purple edges in the construction path from $g_{i,i-k}$, we have that $D^C(g_{i,i-k})(H)$ is a subgraph of $P(G_n)$.
\end{enumerate}
\end{lem}

In \cite{pII} Lemma 3.4 it is proved that, under certain conditions, given a path in an admissible ltt structure starting with the red edge and living in the construction subgraph (we call such a path a \emph{potential construction path}), the process of Example \ref{E:AutomorphismFromPath} yields an admissible construction composition with that path as its construction path. Hence, one can look for construction compositions via paths in construction subgraphs. We do not make this statement precise here because, as we only deal with specific low-rank examples in this paper, one can check by hand in each particular circumstance whether one has obtained a construction composition via the procedure of Example \ref{E:AutomorphismFromPath}, then apply Lemma \ref{L:Building}.  

\vskip10pt

\subsection{Switch sequences}{\label{ss:SwitchSequences}}

An important tool we use to ensure our loops ``close up'' (in circumstances where we do not start with a loop in an $\id$ diagram) is a ``switch sequence.'' As with construction compositions, they can be found via paths in ltt structures and one still needs to check that every triple constructed from the path is admissible (in particular that the ltt structures are admissible).

\begin{df}
An admissible \emph{switch sequence} for a connected, $(2r-1)$-vertex graph $\mathcal{G}$ is an admissible composition $(g_{i-k}, \dots, g_i; G_{i-k-1}, \dots, G_i)$ for $\mathcal{G}$ such that
\indent{\begin{description}
\item  [(SS1)] each  $(g_j; G_{j-1}, G_j)$ with $i-k \leq j \leq i$ is a switch and
\item  [(SS2)] $a(n+1)=u(n) \neq u(l)=a(l+1)$ and $\overline{e_{a(l)}} \neq
e_{u(n)}=e_{a(n+1)}$ for each $i \geq n > l \geq i-k$.
\end{description}}
\end{df}

\begin{rk}
While (ss2) is technical, it is very important because the purple subgraph of an ltt structure changes with a switch and so, after applying a sequence of switches, the purple edge that would have determined the next switch may no longer exist in the ltt structure.
\end{rk}

``Switch paths'' are handled more formally in \cite{pII} than we need here, as we can use the ideas behind them to suggest the construction of a sequence, that we can then check by hand is an admissible switch sequence. We describe the idea behind the paths here: 

Consider an admissible switch sequence $(g_{i-k}, \dots, g_i; G_{i-k-1}, \dots, G_i)$. Recall that the red edge in each $G_j$ is [$e_{j,u(j)}, \overline{e_{j,a(j)}}$]. Because of the nature of switches, under the constraints of (ss2), each [$e_{i,u(j)}, \overline{e_{i,a(j)}}$] is a purple edge $G_i$. Since a switch $(g_j; G_{j-1}, G_j)$ satisfies that $u(j-1)=a(j)$, these purple edges piece together (with black edges between) to form a smooth path in $G_i$ determined by the sequence of vertices: $e_{i,u(i)}$, $\overline{e_{i,a(i)}}$, $e_{i,a(i)}=e_{i,u(i-1)}$, $\overline{e_{i,a(i-1)}}$, $e_{i,a(i-1)}$, $\dots$, $\overline{e_{i,a(i-k+1)}}$, $e_{i,a(i-k+1)}=e_{i,u(i-k)}$, $\overline{e_{i,a(i-k)}}$, $e_{i,a(i-k)}.$

We call this smooth path the \emph{switch path} for $(g_{i-k}, \dots, g_i; G_{i-k-1}, \dots, G_i)$. It is proved in \cite{pII}: If $(g_1, \dots, g_n, G_0, \dots, G_n)$ is an i.d. tt representative for a $\phar$, with $\iwp$ a connected, $(2r-1)$-vertex graph, and $(g_{i-k}, \dots, g_i; G_{i-k-1} \dots, G_i)$ is a switch sequence for some $1 \leq k < i \leq n$, then the associated switch path forms a smooth path in $G_i$.

Thus, one can start with a smooth path and check that it gives an admissible sequence of switches. It is worth noting, though, that (ss2) is still necessary to ensure that the purple edges in a switch path do not disappear before they can play their role as the determining edge for the switch.

\begin{ex}{\label{E:SwitchPath}} In the ltt structure $G_i$, we number the colored edges of a switch path:
~\\
\vspace{-6mm}
\begin{figure}[H]
\centering
\includegraphics[width=1in]{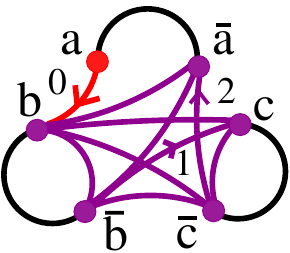}
\label{fig:SwitchPath} \\[-3mm]
\end{figure}

\noindent The switch sequence constructed from the switch path is:
~\\
\vspace{-7mm}
\begin{figure}[H]
\centering
\noindent \includegraphics[width=4.5in]{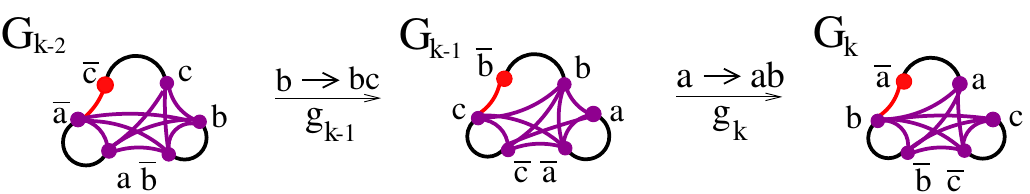}
\label{fig:SwitchSequence} \\[-3mm]
\end{figure}

The red edge in $G_k$ is (0), the red edge in $G_{k-1}$ is (1), and the red edge in $G_{k-2}$ is (2). Notice that here the purple graph changes according to the isomorphism in (swI).
\end{ex}

\subsection{Strategy I}{\label{ss:ss1}}

\begin{description}
\item [Step 1] Find a loop $L(g_1, \dots, g_n; G_0, G_1 \dots,  G_n)$ in $\idg$ so that, for each $1 \leq i \leq r$, there exists a $G_j$ such that either $E_i$ or $\ol{E_i}$ labels the red vertex in $G_j$. \newline
\indent (This can be accomplished using the $\id$ diagram itself or using a switch sequence). 
\item [Step 2] Add in construction compositions or more general loops to ensure (st2). If more general loops are added, one can check as in Example \ref{E:GraphBuilding} whether all of $\mG$ is showing up in the purple subgraphs.
\item [Step 3] Check (st3), then add in loops until irreducibility is obtained.
\item [Step 4] Take a power to ensure all periodic directions are fixed. (This may be necessary since symmetries in $\mG$ may allow a loop in $\idg$ to return to the same ltt structure node without the constructed automorphism having its periodic directions fixed.)
\item [Step 5] Check for the existence of pNP's (see Remark \ref{r:Checks}).
\end{description}

\begin{ex}{\label{E:StrategyI}}  We analyze Graph XX, starting with the switch sequence of Example \ref{E:SwitchPath}.

Our first construction composition (with construction automorphism $a \mapsto ab\bar{c}\bar{c}bbcb$) is given by the construction path in the following ltt structure:
~\\
\vspace{-5.5mm}
\begin{figure}[H]
\centering
\includegraphics[width=1.1in]{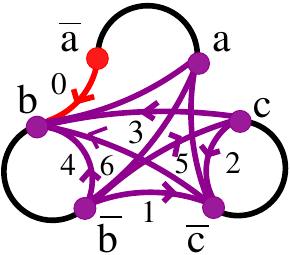}
\label{fig:ConstructionPathExample} \\[-8.5mm]
\end{figure}.

\noindent After that composition we still need:
~\\
\vspace{-6.5mm}
\begin{figure}[H]
\centering
\includegraphics[width=.8in]{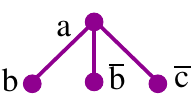}
\label{fig:WhatLeft} \\[-2.5mm]
\end{figure}

\noindent We take the preimage of edges left under the direction map for the final switch and get:
~\\
\vspace{-7.5mm}
\begin{figure}[H]
\centering
\includegraphics[width=2.1in]{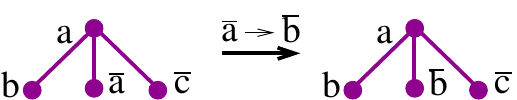}
\label{fig:preimage} \\[-2.5mm]
\end{figure}

\noindent Since we could not obtain all these edges from a single construction composition, we take another preimage (the preimage under the direction map of a second switch in the switch sequence):
~\\
\vspace{-7mm}
\begin{figure}[H]
\centering
\includegraphics[width=2.1in]{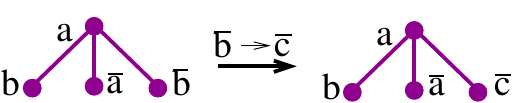}
\label{fig:anotherpreimage} \\[-2.5mm]
\end{figure}

\noindent We use the construction composition for the following construction path to obtain these edges:
~\\
\vspace{-7mm}
\begin{figure}[H]
\centering
\includegraphics[width=1.2in]{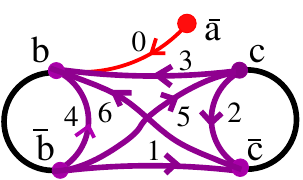}
\label{fig:constructionpath} \\[-3.5mm]
\end{figure}

\noindent When composed we get:
\begin{figure}[H]
\centering
\includegraphics[width=5.15in]{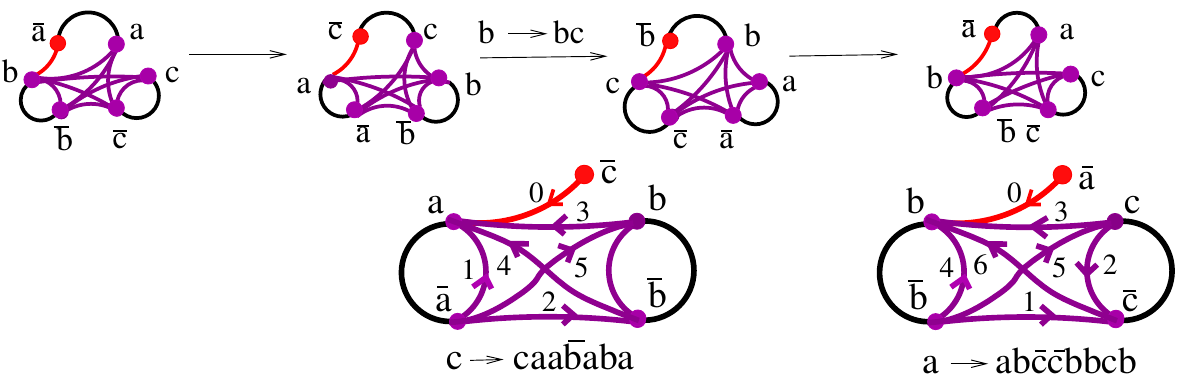}
\label{fig:ConstructionCompositions} 
\end{figure}

\noindent The automorphism obtained is:

$$
\Phi =
\begin{cases} a \mapsto ab \bar{c} \bar{c} bbcb \\
b \mapsto bc \\
c \mapsto cab \bar{c} \bar{c} bbcbab \bar{c} \bar{c} bbcb \bar{c} \bar{c} \bar{b} ab \bar{c} \bar{c} bbcbbccab \bar{c} \bar{c} bbcb
\end{cases}
$$

\noindent Since the periodic directions for this map are not fixed, we take $\Phi^2$, so that the representative is $g_{\Phi^2}$.
\end{ex}

\vskip10pt

\subsection{Strategy II}{\label{ss:ss1}}

In cases where most ``determining'' edges yield an admissible switch and an admissible extension, the $\id$ diagram can be large and impractical to construct. On the other hand, in these cases, a potential construction path is likely to lead to a construction composition and a potential switch path is likely to lead to a switch sequence. We thus present a second strategy (``Strategy II'') for this circumstance.

\begin{df}[Preimage subgraph]{\label{d:PreimageSubgraph}} For an admissible map ($g_{(k,m)}$; $G_{m-1}$, \dots, $G_k$), the \emph{preimage subgraph} under ($g_{(k,m)}$; $G_{m-1}$, \dots, $G_k$) for a subgraph $H \subset P(G_i)$ is obtained from $H$ by replacing each edge of $H$ with its preimage under the isomorphism from $P(G_{m-1})$ to $P(G_k)$.
\end{df}

\noindent The following is \emph{Strategy II}:
\begin{description}
\item [Step 1] Choose an admissible ltt structure G for $\mG$, determine $G_{con}$, and find a potential construction path $p$ in $G_{con}$ (hopefully traversing as many distinct purple edges as possible).
\item [Step 2] Carry out the procedure of Example \ref{E:AutomorphismFromPath} and check whether this gives a construction composition.
\item [Step 3] Repeat Step 1 and Step 2 until a construction composition is obtained from a construction path $p$.
\item [Step 4] Let $H$ denote the subgraph of $G$ consisting of purple edges not hit by $p$, i.e. $H=\mP(G)-\mP(G)\cap p$.
\item [Step 5] Perform admissible switches and extensions (each time taking the preimage of $H$, as in Definition \ref{d:PreimageSubgraph}) until an admissible ltt structure $G'$ is reached whose construction subgraph $G'_{con}$ contains a path $p'$ with edges from the preimage of $H$.
\item [Step 6] Recursively apply Steps 1-5 to the new $G'$ and $p'$ until there are no edges left to construct.
\item [Step 7] Close up the loop either by a switch sequence ending with $G$ or by applying admissible moves until one reaches $G$. (One can note that this process is in fact finite, as each ltt structure only has finitely many potential ``determining'' edges, each ``determining'' edge can lead to at most one admissible switch and one admissible extension, and there are only finitely many admissible ltt structures for a given connected, $(2r-1)$-vertex graph $\mG$.)
\item [Step 8] Take a power to ensure all periodic directions are fixed.
\item [Step 9] Check for the existence of pNP's (see Remark \ref{r:Checks}).

\begin{ex}{\label{E:StrategyI}} 
We show how to apply Strategy II to obtain Graph XIII. It should be noted that we add an extra edge in the first construction path so that, before a switch is chosen for the construction composition, the initial and final ltt structures (the first and last of the five graphs depicted below) are the same. While unnecessary for the procedure, this makes recording the example in short-hand simpler.

\noindent \includegraphics[width=5.5in]{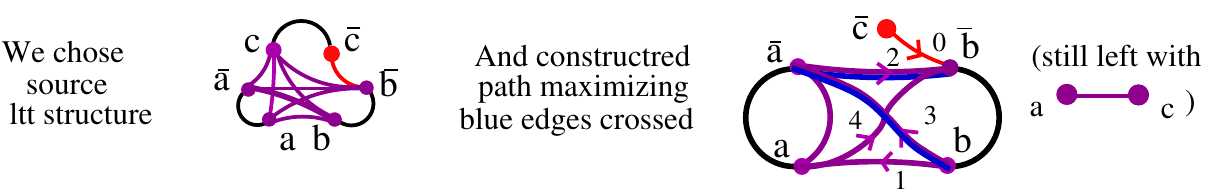} 

\noindent \includegraphics[width=5.3in]{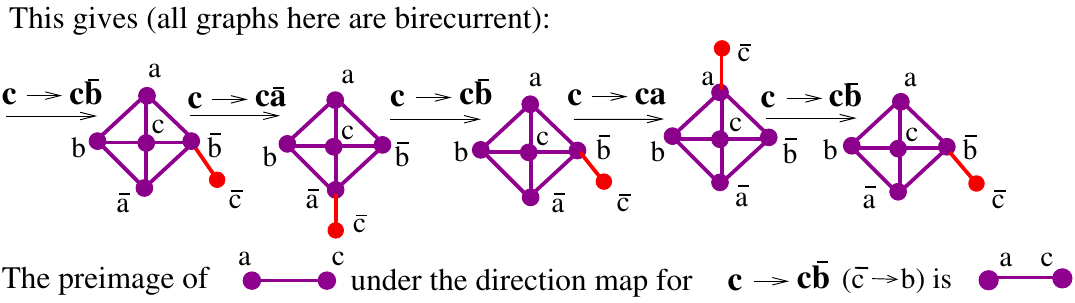}

\vskip5pt

\noindent \includegraphics[width=6.2in]{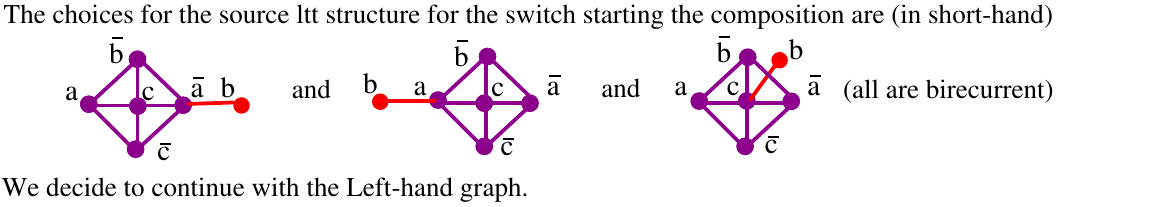}

\vskip5pt

\noindent \includegraphics[width=5.5in]{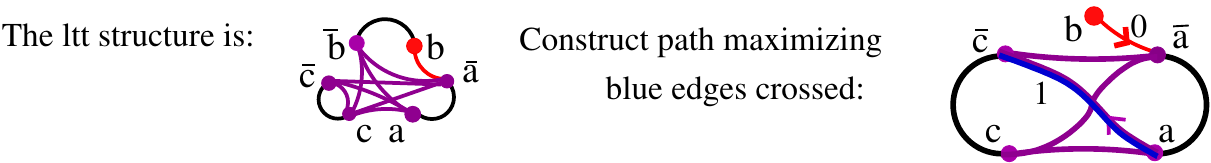}

\vskip5pt

\noindent \includegraphics[width=4.8in]{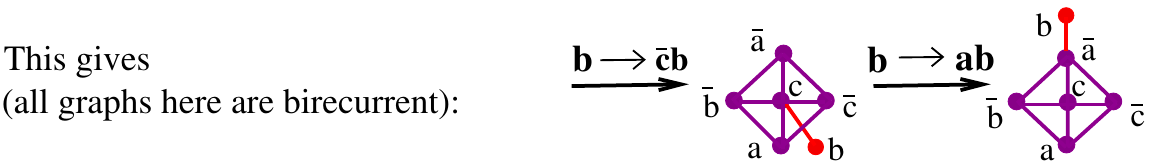}

\noindent We close up the loop with:

~\\
\vspace{-8mm}
\begin{figure}[H]
\centering
\noindent \includegraphics[width=3in]{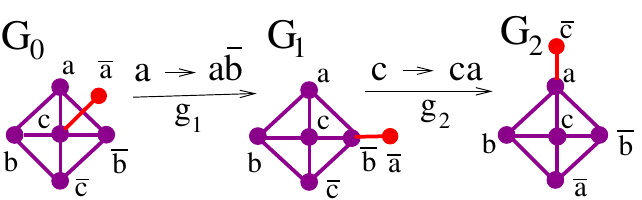}
\label{fig:FinalGenerators} 
\end{figure}

\noindent We have the final map and get the entire representative for Graph XIII:
~\\
\vspace{-5mm}
\noindent \begin{figure}[H]
\centering
\includegraphics[width=5.5in]{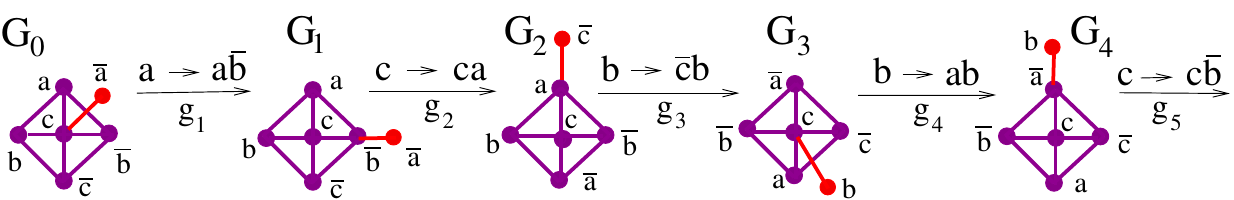}
\label{fig:NPAlgorithmExample1}
\end{figure}
~\\
\vspace{-10mm}
\noindent \begin{figure}[H]
\centering
\includegraphics[width=5.5in]{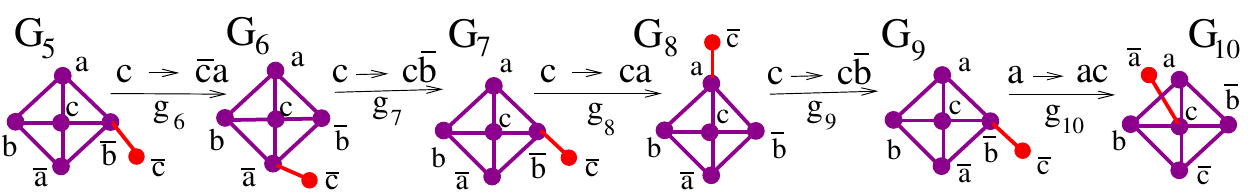}
\label{fig:NPAlgorithmExample2}
\end{figure}
~\\
\vspace{-10mm}
\noindent \begin{figure}[H]
\centering
\includegraphics[width=5.5in]{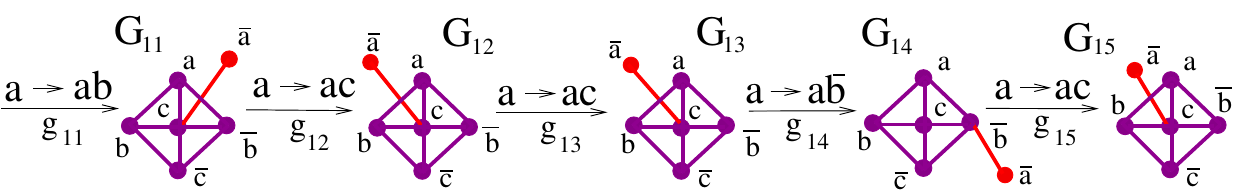}
\label{fig:NPAlgorithmExample3} 
\end{figure}

\noindent It can be noted that we showed that this map does not have any pNPs in \cite{pI}. 
\end{ex}

\end{description}

\vskip5pt

\section{Achievable Graphs in Rank 3}{\label{Ch:Achievable}}

This section includes our main theorem. The theorem gives a refinement of the achievability of the index list $(-\frac{3}{2})$ by fully irreducible $\phi \in Out(F_3)$.

\begin{thm}[Theorem A]{\label{T:MainTheorem}}  Precisely eighteen of the twenty-one connected, simplicial five-vertex graphs are the ideal Whitehead graph $\mathcal{IW}(\phi)$ for a fully irreducible outer automorphism $\phi\in Out(F_3)$. \end{thm}

\begin{proof} Graphs II, V, and VII were proved unachievable in \cite{pI}. We give representatives for the remaining graphs, leaving it to the reader to prove they are pNP-free (see Remark \ref{r:Checks}), that they satisfy Lemma \ref{l:RepresentativeLoops}B, and have the appropriate ideal Whitehead graphs. Then, by Lemma \ref{l:RepresentativeLoops}, they are representatives of $\phi \in \mathcal{A}_r$ with the desired ideal Whitehead graphs.

For each achieved graph, we give a representative $g$ achieving it and then an ideal decomposition for $g$. When showing the ideal decomposition, in most cases, we leave out the black edges in the ltt structures. For Graphs X, XII, XV, and XIX we give a condensed description of the ideal decomposition where a pure construction composition starting and ending at an ltt structure $G_i$ is shown below as a path in $(G_i)_{con}$. For graphs XI and XVI, the pure construction compositions do not start and end with the same ltt structure, so are depicted as paths in $(G_i)_{con}$ below, but between, their initial and terminal ltt structures.

\smallskip

\noindent Graph I (The Line):
$$
\Phi =
\begin{cases} a \mapsto ac \bar{b} ca \bar{b} cacac \bar{b} ca  \\
b \mapsto \bar{a} \bar{c} b \bar{c} \bar{a} \bar{c} \bar{a} \bar{c} b  \\
c \mapsto cac \bar{b} ca \bar{b} cac
\end{cases}
$$
\begin{figure}[H]
\centering
\noindent \includegraphics[width=6.5in]{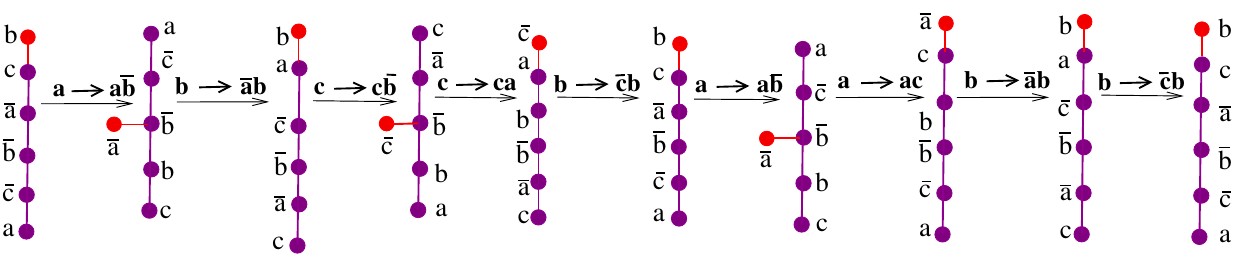}
\end{figure}

\noindent Graph III:
$$
\Phi =
\begin{cases} a \mapsto a \bar{b} ca  \\
b \mapsto b \bar{a} \bar{c} \bar{a} \bar{c} \bar{c} \bar{a} \bar{c}  \\
c \mapsto caccaca \bar{b} cac
\end{cases}
$$
\begin{figure}[H]
\centering
\noindent \includegraphics[width=6.2in]{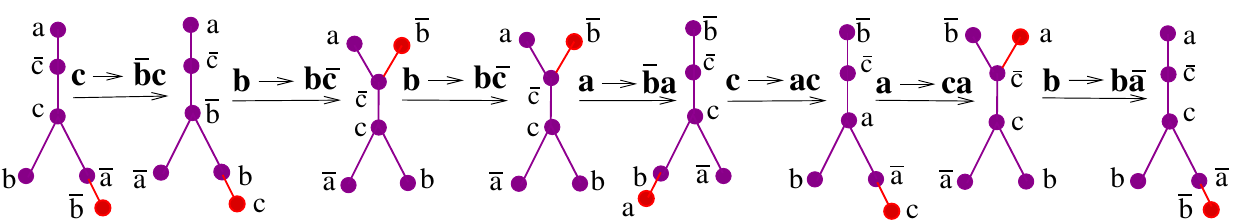}
\end{figure}

\noindent Graph IV:
$$
\Phi =
\begin{cases} a \mapsto c \bar{b} a  \\
b \mapsto bc \bar{a} bcb  \\
c \mapsto c \bar{b} abc
\end{cases}
$$
\begin{figure}[H]
\centering
\noindent  \includegraphics[width=6in]{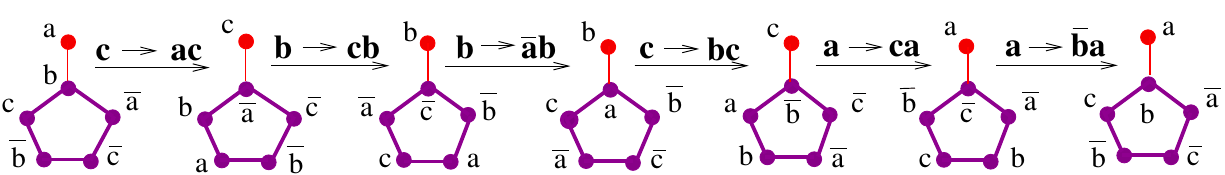}
\end{figure}

\noindent Graph VI:
$$
\Phi =
\begin{cases} a \mapsto abacbaba \bar{c} abacbaba  \\
b \mapsto ba \bar{c}  \\
c \mapsto c \bar{a} \bar{b} \bar{a} \bar{b} \bar{a} \bar{b} \bar{c} \bar{a} \bar{b} \bar{a} c  \end{cases}
$$
\begin{figure}[H]
\centering
\noindent \includegraphics[width=6.85in]{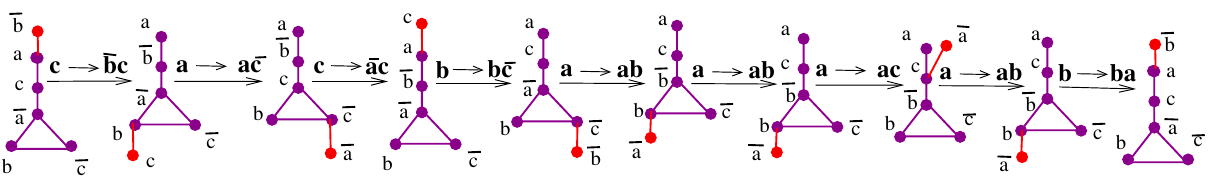}
\end{figure}

\noindent Graph VIII:
$$
\Phi =
\begin{cases} a \mapsto a \bar{c} aa \bar{b} a \bar{c} b \bar{a} \bar{a} ca \bar{c} aa \bar{b} a \bar{c} a  \\
b \mapsto b \bar{a} \bar{a} c  \\
c \mapsto c \bar{a} b \bar{a} \bar{a} c \bar{a} b \bar{a} \bar{a} c
\end{cases}
$$
\begin{figure}[H]
\centering
\noindent \includegraphics[width=4.4in]{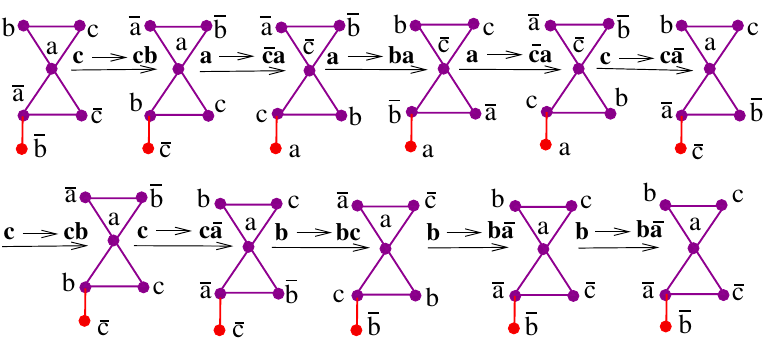}
\end{figure}

\noindent Graph IX:
$$
\Phi =
\begin{cases} a \mapsto ab \bar{c} b \bar{c} ab \bar{c}b \bar{c} \bar{b} c \bar{b} \bar{a}\bar{c} \bar{b} ab \bar{c} b \bar{c} \\
b \mapsto bcab \bar{c} bc \bar{b} c \bar{b} \bar{a} cab \bar{c} b  \\
c \mapsto c \bar{b} c \bar{b} \bar{a} bcab \bar{c} bc \bar{b} c \bar{b} \bar{a} c
\end{cases}
$$
\begin{figure}[H]
\centering
\noindent \includegraphics[width=6.5in]{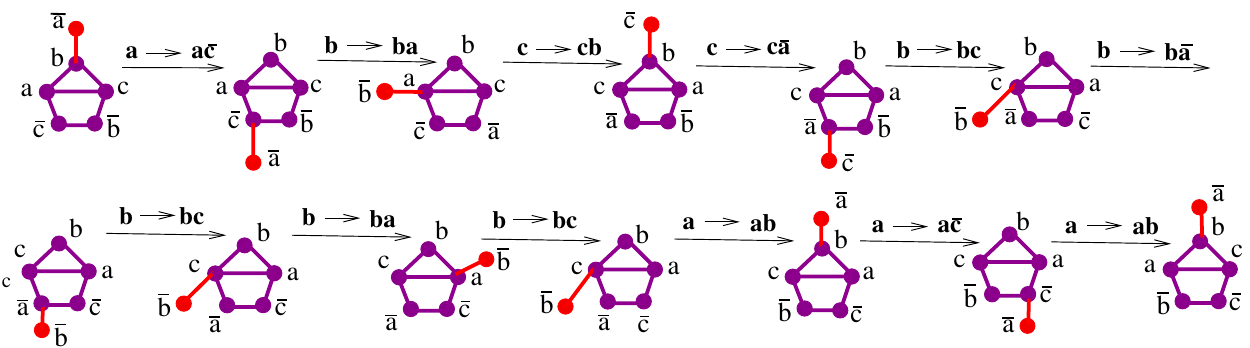}
\end{figure}

\noindent Graph X:
$$
\Phi =
\begin{cases} a \mapsto abacbabac \bar{a} b \bar{c} \bar{a} \bar{b} \bar{a} \bar{c} \bar{a} \bar{b} \bar{a} \bar{b} \bar{c} \bar{a} \bar{b} \bar{a} babac \bar{b} abacbabacabac \bar{b} a  \\
b \mapsto babac \bar{a} b \bar{c} \bar{a} \bar{b} \bar{a} \bar{c} \bar{a} \bar{b} \bar{a} \bar{b} \bar{c} \bar{a} \bar{b} \bar{a} babac \bar{a} b \bar{c} \bar{a} \bar{b} \bar{a} \bar{c} \bar{a} \bar{b} \bar{a} \bar{b} \bar{c} \bar{a} \bar{b} \bar{a} b  \\
c \mapsto babac \bar{a} b \bar{c} \bar{a} \bar{b} \bar{a} \bar{c} \bar{a} \bar{b} \bar{a} \bar{b} \bar{c} \bar{a} \bar{b} \bar{a} babac \bar{a} b \bar{c} \bar{a} \bar{b} \bar{a} \bar{c} \bar{a} \bar{b} \bar{a} \bar{b} \bar{c} \bar{a} \bar{b} \bar{a} babacbabac \bar{a} b \bar{c} \bar{a} \bar{b} \bar{a} \bar{c} \bar{a} \bar{b} \bar{a} \bar{b} \bar{c} \bar{a} \bar{b} \bar{a} babac
\end{cases}
$$
\begin{figure}[H]
\centering
\noindent \includegraphics[width=6.2in]{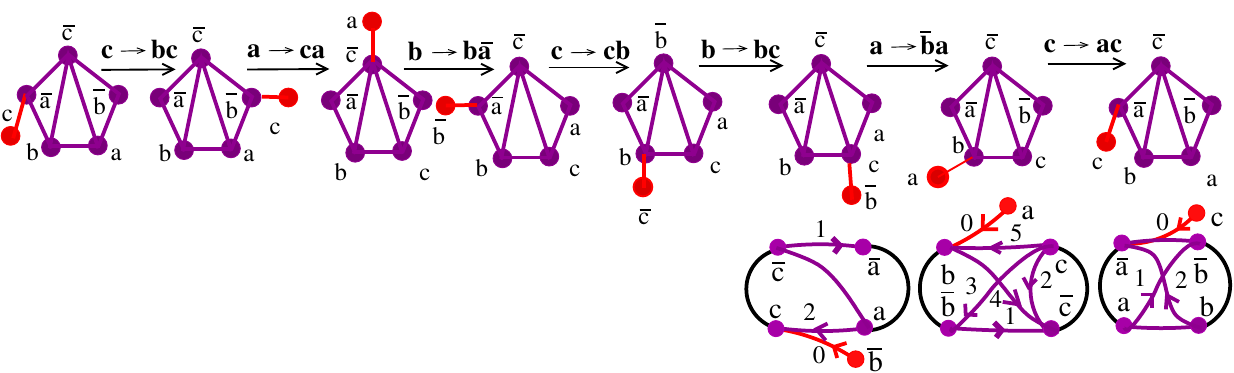}
\end{figure}

\noindent Graph XI:
$$
\Phi =
\begin{cases} a \mapsto a \bar{c} \bar{b} \bar{c} b \bar{c} \bar{b} c \bar{b} cbc \bar{a} bc \bar{b} cbc \bar{a} \bar{b} c \bar{b}  \\
b \mapsto b \bar{c} ba \bar{c} \bar{b} \bar{c} b \bar{c} \bar{b} a \bar{c} \bar{b} \bar{c} b \bar{c} b  \\
c \mapsto c \bar{b} cbc \bar{a} \bar{b} c \bar{b} cbc \bar{a} bc \bar{b} cbc \bar{a} \bar{b} c  \end{cases}
$$
\begin{figure}[H]
\centering
\noindent \includegraphics[width=5.8in]{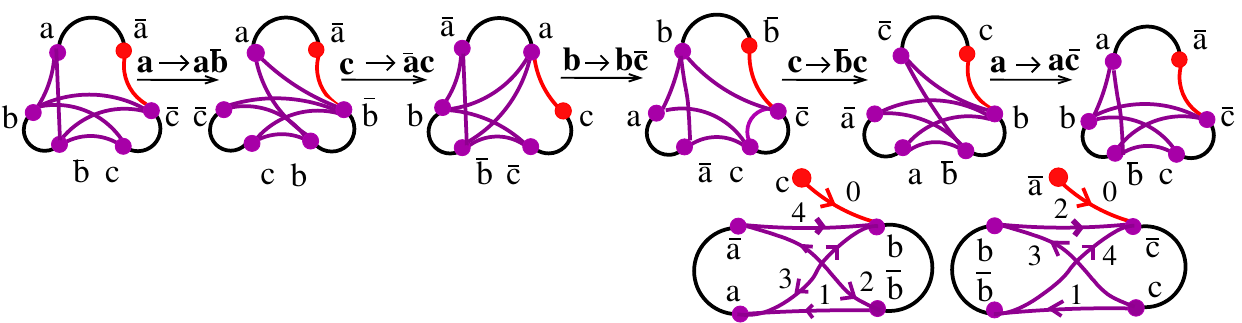}
\end{figure}

\noindent Graph XII:
$$
\Phi =
\begin{cases} a \mapsto a \bar{c} \bar{b} \bar{b} \bar{c} b \bar{c} \bar{b} c \bar{b} cbcc \bar{a} bc \bar{b} cbcc \bar{a} \bar{b} c \bar{b}  \\
b \mapsto b \bar{c} ba \bar{c}\bar{c} \bar{b} \bar{c} b \bar{c} \bar{b} a \bar{c} \bar{c} \bar{b} \bar{c} b \bar{c} b  \\
c \mapsto c \bar{b} cbcc \bar{a} \bar{b} c \bar{b} cbcc \bar{a} bc \bar{b} cbcc \bar{a} \bar{b} c  \end{cases}
$$
\begin{figure}[H]
\centering
\noindent \includegraphics[width=6.5in]{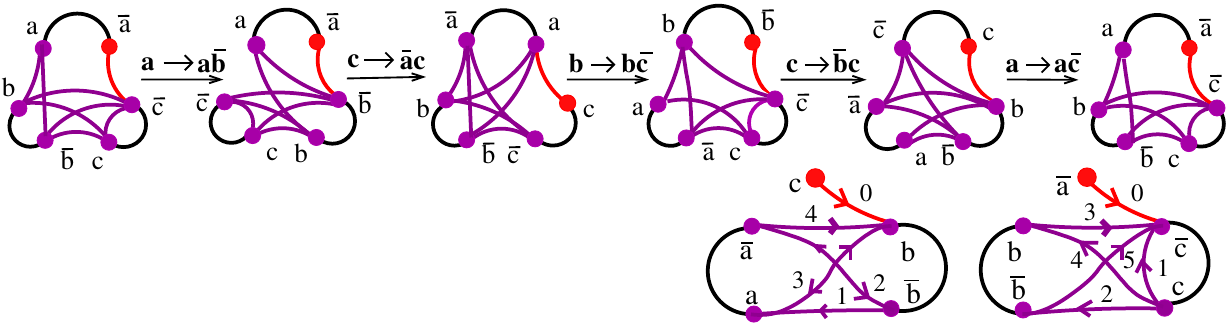}
\end{figure}

\noindent Graph XIII:
$$
\Phi =
\begin{cases} a \mapsto ac \bar{b} ccbc \bar{b} \bar{c} \bar{b} \bar{c} \bar{c} b \bar{c} \bar{a} c \bar{b} ac \bar{b} ccbc \bar{b} \bar{c} \bar{b} \bar{c} \bar{c} b \bar{c} \bar{a} \bar{b}  \\
b \mapsto bac \bar{b} ccbcb \bar{c} \bar{b} \bar{c} \bar{c} b \bar{c} \bar{a} b \bar{c} ac \bar{b} ccbcb  \\
c \mapsto c \bar{b} ac \bar{b} ccbc \bar{b} \bar{c} \bar{b} \bar{c} \bar{c} b \bar{c} \bar{a} \bar{b} ac \bar{b} ccbc
\end{cases}
$$

\smallskip

\noindent Our ideal decomposition for this representative and further explanation were given in Example \ref{E:GraphXIII}.

\vskip15pt

\noindent Graph XIV:
$$
\Phi =
\begin{cases} a \mapsto abcaabca\bar{b}abcaca\bar{b}abcaabcaabca
\bar{b}abcaca\bar{b}abcaabcaabca\bar{b}abcaabcaabca\bar{b}abca
ca\bar{b}abcab\bar{a}\bar{c}abc  \\
aababcaabca\bar{b}abcaca\bar{b}abcaa
bcaabca\bar{b}abcaca\bar{b}
abcaabcaabca\bar{b}  
abcaabcaabca\bar{b}abcaca\bar{b}abca  \\
b \mapsto b\bar{a}\bar{c}abcaab\bar{a}\bar{c}
\bar{b}\bar{a}b\bar{a}\bar{c}\bar{b}
\bar{a}\bar{a}\bar{c}\bar{b}\bar{a}\bar{a}\bar{c} 
\bar{b}\bar{a}b\bar{a}\bar{c}\bar{a}\bar{c}\bar{b}\bar{a}b
\bar{a}\bar{c}\bar{b}\bar{a}\bar{a}  
\bar{c}\bar{b}\bar{a}\bar{a}\bar{c}\bar{b}\bar{a}b\bar{a}
\bar{c}\bar{b}\bar{a}\bar{a}\bar{c}\bar{b}\bar{a}\bar{a}\bar{c}
\bar{b}\bar{a}b\bar{a}\bar{c}abcaabca\bar{b}abcaca\bar{b}  \\
abcaabcaabca\bar{b}abcaca\bar{b}abcaabcaabca\bar{b}abcaabcaa
bca\bar{b}abcaca\bar{b}abcab\bar{a}\bar{c}abcaab  \\
c \mapsto ca\bar{b}abcaabcaabca\bar{b}abcaabcaabca
\bar{b}abcaca\bar{b}abcaa
bcaabca\bar{b}abca\bar{b}\bar{a}\bar{a}\bar{c}\bar{b}\bar{a}ca\bar{b}
\end{cases}
$$
\begin{figure}[H]
\centering
\noindent \includegraphics[width=6.5in]{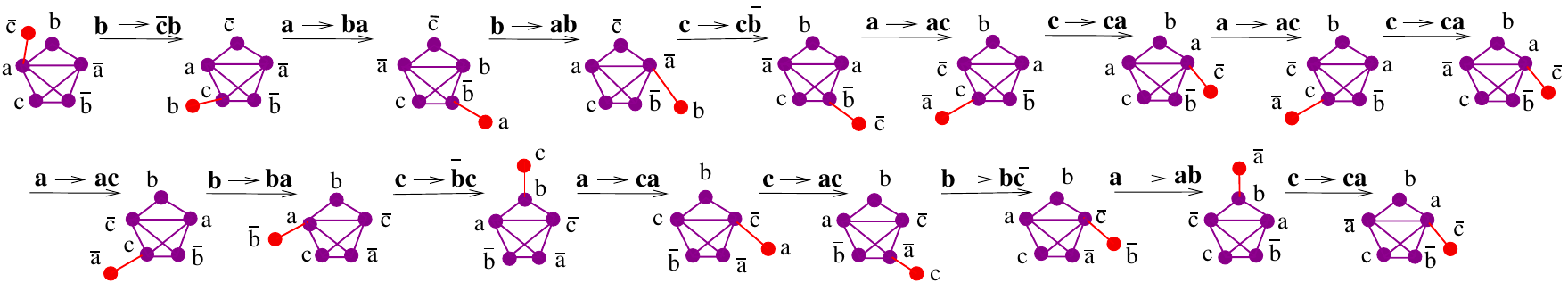}
\end{figure}

\noindent Graph XV:
$$
\Phi =
\begin{cases} a \mapsto a \bar{c} \bar{b} \bar{b} \bar{c} b \bar{c} \bar{b} c \bar{b} cbbc \bar{a} bc \bar{b} cbbc \bar{a} \bar{b} c \bar{b}  \\
b \mapsto b \bar{c} ba \bar{c} \bar{b} \bar{b} \bar{c} b \bar{c} \bar{b} a \bar{c} \bar{b} \bar{b}  \bar{c} b \bar{c} b  \\
c \mapsto c \bar{b} cbbc \bar{a} \bar{b} c \bar{b} cbbc \bar{a} bc \bar{b} cbbc \bar{a} bc  \end{cases}
$$
\begin{figure}[H]
\centering
\noindent \includegraphics[width=6in]{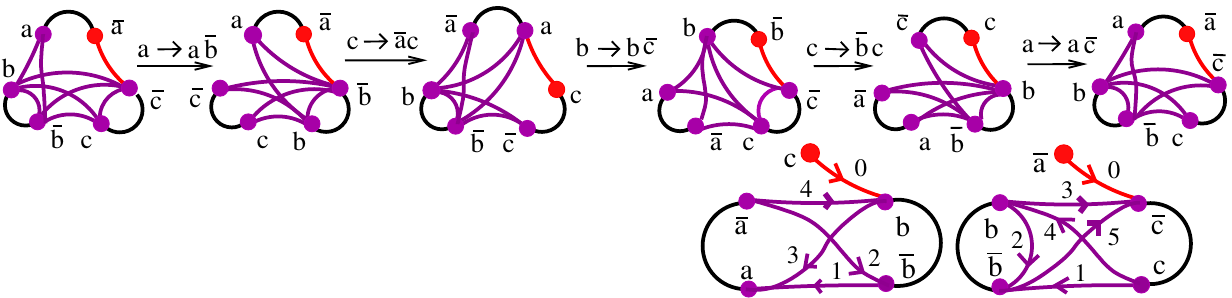}
\end{figure}

\noindent Graph XVI:
$$
\Phi =
\begin{cases} a \mapsto a \bar{b} ccbc \bar{b} c  \\
b \mapsto b \bar{c} \bar{b} \bar{c} \bar{c} b \bar{a} \bar{c} b  \\
c \mapsto ca \bar{b} ccbc \bar{b} c
\end{cases}
$$
\begin{figure}[H]
\centering
\noindent \includegraphics[width=5in]{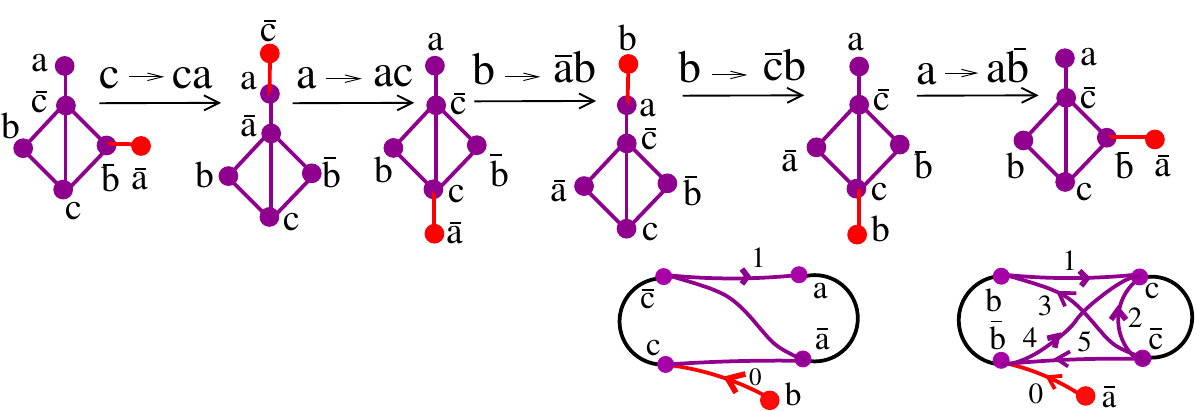}
\end{figure}

\noindent Graph XVII:
$$
\Phi =
\begin{cases} a \mapsto acbc\bar{b} c\bar{b} acbc\bar{b} acbc  \\
b \mapsto b\bar{c} \bar{b} \bar{c} \bar{a} b\bar{c} \bar{c} \bar{b} \bar{c} \bar{a} b\bar{c} \bar{b} \bar{c} \bar{a} b\bar{c} b  \\
c \mapsto c\bar{b} acbc\bar{b} \bar{b} c\bar{b} acbc\bar{b} acbcc\bar{b} acbc\bar{b} acbc\bar{b} c\bar{b} acbc\bar{b} acbc
\end{cases}
$$
\begin{figure}[H]
\centering
\noindent \noindent \includegraphics[width=5.6in]{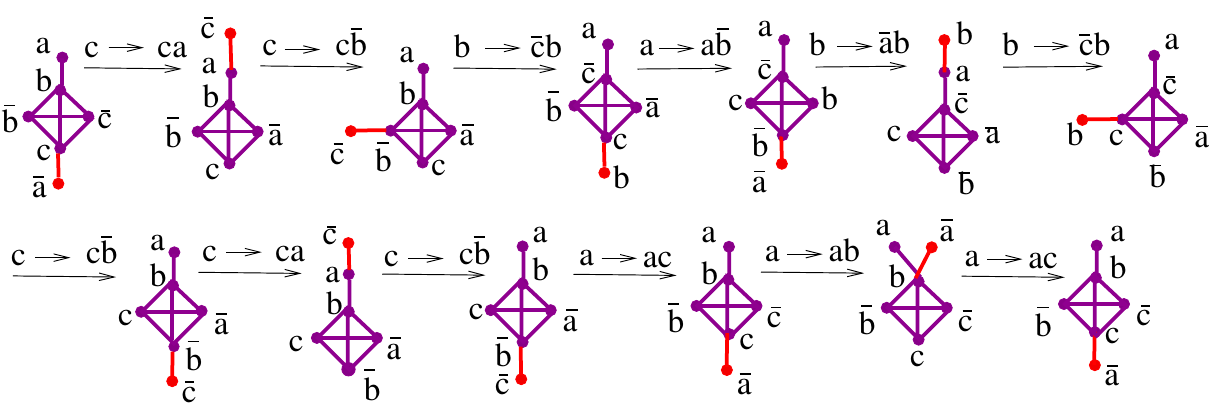}
\end{figure}

\noindent Graph XVIII:
$$
\Phi =
\begin{cases} a \mapsto a \bar{b} c \bar{b} a \bar{b}  \bar{c}  \\
b \mapsto b \bar{a} b  \bar{c} b \bar{a} b  \bar{a} b  \\
c \mapsto cb \bar{a} b \bar{c} b \bar{a} b \bar{a} b \bar{c} b \bar{a} b \bar{a} bc
\end{cases}
$$
\begin{figure}[H]
\centering
\noindent \includegraphics[width=5.5in]{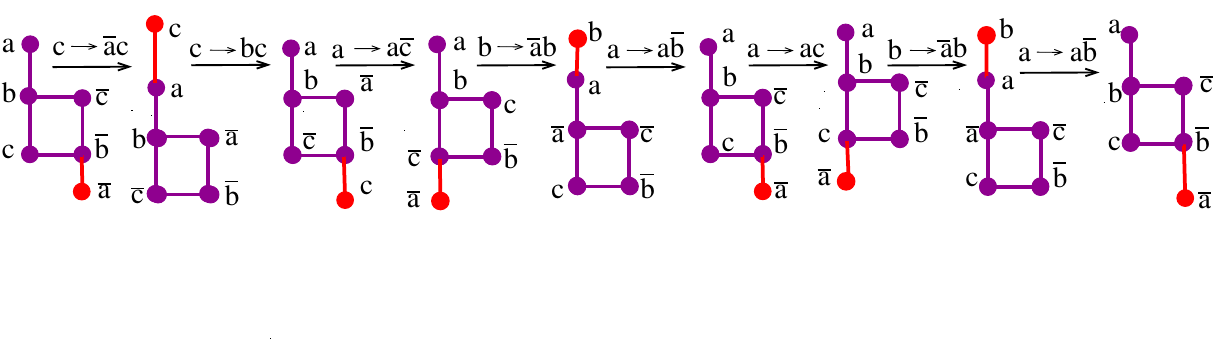} 
\end{figure}

\noindent Graph XIX:
$$
\Phi =
\begin{cases} a \mapsto acc \bar{b} cbc  \\
b \mapsto b \bar{c} \bar{b} \bar{c} b \bar{c} \bar{c} \bar{a} b \bar{c} b  \\
c \mapsto c \bar{b} acc \bar{b} cbc \bar{b} acc \bar{b} cbc
\end{cases}
$$
\begin{figure}[H]
\centering
\noindent \includegraphics[width=5.2in]{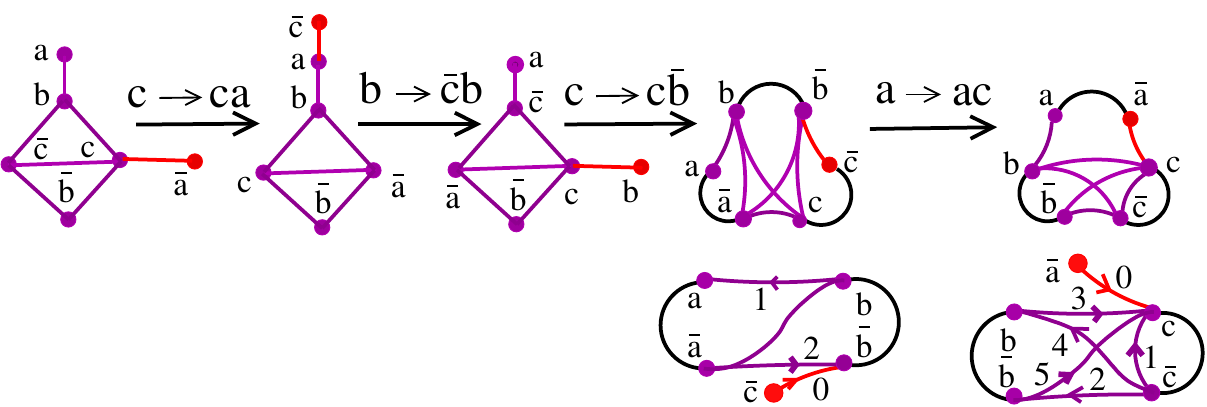}
\end{figure}

\noindent Graph XX:

\vskip7pt

\noindent The representative $g=h_{\Phi}^2$ having ideal Whitehead graph GRAPH XX, where
$$
\Phi =
\begin{cases} a \mapsto ab \bar{c} \bar{c} bbcb  \\
b \mapsto bc  \\
c \mapsto cab \bar{c} \bar{c} bbcbab \bar{c} \bar{c} bbcb \bar{c} \bar{c} \bar{b} ab \bar{c} \bar{c} bbcbbccab \bar{c} \bar{c} bbcb
\end{cases},
$$
was constructed in the examples above.

\vskip15pt

\noindent Graph XXI (Complete Graph): This was given in \cite{pII}.

$$
\Phi =
\begin{cases} a \mapsto aba \bar{b} aac \bar{b} aba \bar{b} aacbaba \bar{b} aacaba \bar{b} aac \bar{b} a  \\
b \mapsto baba \bar{b} aac \bar{a} b \bar{c} \bar{a} \bar{a} b \bar{a} \bar{b} \bar{a} \bar{c} \bar{a} \bar{a} b \bar{a} \bar{b} \bar{a} \bar{b} \bar{c} \bar{a} \bar{a} b \bar{a} \bar{b} \bar{a} b  \\
c \mapsto aba \bar{b} aac
\end{cases}
$$

\vskip15pt

\noindent \qedhere
\end{proof}

\vskip35pt


\bibliographystyle{amsalpha}
\bibliography{BirecurrencyCondition}

\end{document}